\documentclass[11pt]{amsart}

\usepackage{
    amsmath,
    amsfonts,
    amssymb,
    amsthm,
    amscd,
    comment,
    enumitem,
    etoolbox,
    textcomp,
    gensymb,    
    mathtools,
    mathdots,
    stmaryrd,
    graphicx
}
\usepackage[dvipsnames]{xcolor}
\usepackage{tikz-cd}
\usepackage{csquotes}
\usepackage{dsfont}
\usepackage{ytableau}
\usepackage{mathtools}

\usepackage[letterpaper,margin=1in]{geometry}
\usepackage{pgfornament}
\usepackage[x11names]{xcolor}
\usepackage{pgfornament}
\usepackage{adjustbox}

\usepackage{enumitem}

\newenvironment{enum}
  {\begin{enumerate}[label=\arabic*)]}
  {\end{enumerate}}

\usepackage[T1]{fontenc}
\usepackage{lmodern}            
\usepackage[colorlinks=true, linkcolor=blue, citecolor=blue, urlcolor=blue, breaklinks=true]{hyperref}

\usepackage{zref-clever}

\zcsetup{
  cap,                      
  nameinlink=false,         
  lastsep = {, and }        
}

\zcRefTypeSetup{assumption}{
    Name-sg = Assumption ,
    name-sg = assumption ,
    Name-pl = Assumptions ,
    name-pl = assumptions ,
}
\zcRefTypeSetup{cor}{
    Name-sg = Corollary ,
    name-sg = corollary ,
    Name-pl = Corollaries ,
    name-pl = corollaries ,
}
\zcRefTypeSetup{defin}{
    Name-sg = Definition ,
    name-sg = definition ,
    Name-pl = Definitions ,
    name-pl = definitions ,
}
\zcRefTypeSetup{enumi}{
    Name-sg = {} ,
    name-sg = {} ,
    Name-pl = {} ,
    name-pl = {} ,
}
\zcRefTypeSetup{equation}{
    Name-sg = {} ,
    name-sg = {} ,
    Name-pl = {} ,
    name-pl = {} ,
}
\zcRefTypeSetup{eg}{
    Name-sg = Example ,
    name-sg = example ,
    Name-pl = Examples ,
    name-pl = examples ,
}
\zcRefTypeSetup{lem}{
    Name-sg = Lemma ,
    name-sg = lemma ,
    Name-pl = Lemmas ,
    name-pl = lemmas ,
}
\zcRefTypeSetup{prop}{
    Name-sg = Proposition ,
    name-sg = proposition ,
    Name-pl = Propositions ,
    name-pl = propositions ,
}
\zcRefTypeSetup{rem}{
    Name-sg = Remark ,
    name-sg = remark ,
    Name-pl = Remarks ,
    name-pl = remarks ,
}
\zcRefTypeSetup{theo}{
    Name-sg = Theorem ,
    name-sg = theorem ,
    Name-pl = Theorems ,
    name-pl = theorems ,
}

\zcsetup{
  countertype = {
    enumi=item, enumii=item, enumiii=item, enumiv=item
  },
  counterresetby = {
    enumii=enumi, enumiii=enumii, enumiv=enumiii
  }
}
\zcRefTypeSetup{item}{
    Name-sg = {} ,
    name-sg = {} ,
    Name-pl = {} ,
    name-pl = {} ,
}

\AddToHook{env/assumption/begin}{\zcsetup{countertype={theo=assumption}}}
\AddToHook{env/cor/begin}{\zcsetup{countertype={theo=cor}}}
\AddToHook{env/defin/begin}{\zcsetup{countertype={theo=defin}}}
\AddToHook{env/eg/begin}{\zcsetup{countertype={theo=eg}}}
\AddToHook{env/lem/begin}{\zcsetup{countertype={theo=lem}}}
\AddToHook{env/prop/begin}{\zcsetup{countertype={theo=prop}}}
\AddToHook{env/rem/begin}{\zcsetup{countertype={theo=rem}}}

\newcommand{\cref}[1]{\zcref{#1}}
\newcommand{\Cref}[1]{\zcref[S]{#1}}

\makeatletter
\def\namedlabel#1#2{\begingroup
    (#2)%
    \def\@currentlabel{(#2)}%
    \phantomsection\label{#1}\endgroup
}
\makeatother

\newcommand\C{\mathbb{C}}
\newcommand\N{\mathbb{N}}
\newcommand\Q{\mathbb{Q}}

\newcommand\Z{\mathbb{Z}}
\newcommand{\Sc}{\mathbb{S}}

\newcommand{\gS}{\mathfrak{S}}

\definecolor{FieldsGreen}{HTML}{2b7443}

\DeclareMathOperator{\Hom}{Hom}

\DeclareMathOperator{\im}{im}                   

\newcommand{\Par}{\mathrm{Par}}
\newcommand{\SPar}{\mathrm{SPar}}

\newcommand{\SComp}{\mathrm{SComp}}
\newcommand{\Pol}{\mathrm{Pol}}
\newcommand{\LPol}{\mathrm{LPol}}
\newcommand{\Sym}{\mathrm{Sym}}
\newcommand{\LSym}{\mathrm{LSym}}
\newcommand{\LAlt}{\mathrm{LAlt}}
\newcommand{\Alt}{\mathrm{Alt}}

\newcommand{\sgn}{\mathrm{sgn}}
\newcommand{\GL}{\mathrm{GL}}

\usepackage{tikz}
\usetikzlibrary{arrows.meta,patterns}
\usetikzlibrary{shapes.geometric,decorations.markings,decorations.pathreplacing}
\usepackage{tikz-cd}

\tikzset{
    anchorbase/.style={
        >=To,
        line cap = round, line join = round,
        baseline={([yshift=-0.5ex]current bounding box.center)},
    }
}
\tikzset{
    centerzero/.style={
        >=To,
        line cap = round, line join = round,
        baseline={([yshift=-0.5ex]#1)},
        },
    centerzero/.default={0,0}
}

\newtheorem{theo}{Theorem}[section]
\newtheorem{prop}[theo]{Proposition}
\newtheorem{lem}[theo]{Lemma}
\newtheorem{cor}[theo]{Corollary}

\theoremstyle{definition}

\newtheorem{defin}[theo]{Definition}
\newtheorem{rem}[theo]{Remark}
\newtheorem{eg}[theo]{Example}

\numberwithin{equation}{section}
\allowdisplaybreaks

\setenumerate[1]{label=(\alph*)}          

\newtoggle{comments}
\newtoggle{details}
\newtoggle{detailsnote}

\iftoggle{comments}{%
    \usepackage[notref,notcite]{showkeys}   
    \newcommand{\acomments}[1]{
        \ \\
        {\color{red}
            \textbf{AS:} #1
        }
        \ \\
    }
    \newcommand{\ycomments}[1]{
        \ \\
        {\color{red}
            \textbf{YS:} #1
        }
        \ \\
    }
    \newcommand{\hcomments}[1]{
        \ \\
        {\color{cyan}
            \textbf{HF:} #1
        }
        \ \\
    }
    \newcommand{\ecomments}[1]{
        \ \\
        {\color{cyan}
            \textbf{EK:} #1
        }
        \ \\
    }
    \newcommand{\tcomments}[1]{
        \ \\
        {\color{cyan}
            \textbf{TL:} #1
        }
        \ \\
    }
    \newcommand{\mcomments}[1]{
        \ \\
        {\color{cyan}
            \textbf{MTD:} #1
        }
        \ \\
    }
    }{%
        \newcommand{\acomments}[1]{}
        \newcommand{\ycomments}[1]{}
        \newcommand{\hcomments}[1]{}
        \newcommand{\ecomments}[1]{}
        \newcommand{\tcomments}[1]{}
        \newcommand{\mcomments}[1]{}
    }

\definecolor{neongreen}{RGB}{57, 255, 20}

\iftoggle{details}{%
    \newcommand{\details}[1]{
        \ \\
        {\color{OliveGreen}
            \textbf{Details:} #1
        }
        \\
    }
}{%
    \newcommand{\details}[1]{\ignorespaces}
}

\usepackage{wasysym}

\makeatletter

\newlength{\eqalignref@lab}
\newlength{\eqalignref@tmp}

\newcommand{\eqalignref@measure}[1]{%
    \settowidth{\eqalignref@tmp}{%
        \ensuremath{\scriptstyle\text{\cref{#1}}}%
    }%
    \ifdim\eqalignref@tmp>\eqalignref@lab
        \global\eqalignref@lab=\eqalignref@tmp
    \fi
}

\newcommand{\eqalignref@finish}{%
    \settowidth{\eqalignref@tmp}{\ensuremath{=}}%
    \global\advance\eqalignref@lab -\eqalignref@tmp
    \ifdim\eqalignref@lab<\z@
        \global\eqalignref@lab=\z@
    \fi
    \global\eqalignref@lab=.5\eqalignref@lab
}

\newcommand{\eqalignref@pad}[1]{%
    \mspace{#1}\hspace{\eqalignref@lab}%
}

\newcommand{\eqalignref}[2][0mu]{%
    \global\eqalignref@lab=\z@
    \eqalignref@measure{#2}%
    \eqalignref@finish
    \eqalignref@pad{#1}%
    &\mathrel{\mathop{=}\limits^{%
        \mathclap{\text{\cref{#2}}}%
    }}%
    \eqalignref@pad{#1}%
}

\newcommand{\eqalignsref}[3][0mu]{%
    \global\eqalignref@lab=\z@
    \eqalignref@measure{#2}%
    \eqalignref@measure{#3}%
    \eqalignref@finish
    \eqalignref@pad{#1}%
    &\mathrel{\mathop{=}\limits^{%
        \mathclap{\text{\cref{#2}}}%
    }_{%
        \mathclap{\text{\cref{#3}}}%
    }}%
    \eqalignref@pad{#1}%
}
\makeatother

\begin{document}
\title{Laurent symmetric functions} 

\author{Hugo Fernandez}
\address[H.F.]{
    Facultad de CC. Matemáticas \\
    Universidad Complutense de Madrid	 \\
     Pl. de las Ciencias, 3, 28040 Madrid, Spain
}
\email{hugfer01@ucm.es}

\author{Elena Kleinwort}
\address[E.K.]{
    Department of Mathematics\\
    University of Hamburg \\
    20146 Hamburg, Germany
}
\email{elena.kleinwort02@gmail.com}

\author{Tianze Li}
\address[T.L.]{
    Department of Mathematics \\
    University of California \\
    Berkeley, CA 94720, United States
}
\email{tianzeli@berkeley.edu}

\author{Mario Torres Danta}
\address[M.T.D.]{
    Facultad de Matemáticas \\
     Universidad de Sevilla \\
   Calle Tarfia, 41012 Sevilla, Spain
}
\email{mtdanta@outlook.com}

\ifboolexpr{togl{comments} or togl{details}}{%
  {\color{magenta}DETAILS OR COMMENTS ON}
}{
}

\begin{abstract}
    We describe the ring of \textit{Laurent symmetric polynomials} in terms of generators and relations, giving an analogue of the Fundamental Theorem of Symmetric  Polynomials. We extend the Hall inner product, and give an algebraic proof that the \textit{Laurent Schur polynomials} form an orthonormal basis of this ring. In the case of infinitely many variables, we realize the ring of \textit{Laurent symmetric functions} as a direct limit of inverse limits, and relate them to characters of rational and algebraic representations of general linear groups.
\end{abstract}

\maketitle
\thispagestyle{empty}

\tableofcontents

\section{Introduction}
Laurent symmetric functions are generalizations of classical symmetric functions that allow variables with both positive and negative powers. These functions arise naturally in the representation theory of general linear groups. In particular, \textit{Laurent Schur polynomials} are the characters of finite dimensional irreducible rational representations of $\GL_n$, just as ordinary Schur polynomials are the characters of irreducible polynomial representations.

Despite their ubiquity, the theory of Laurent symmetric functions remains relatively unexplored. While Laurent symmetric functions have appeared in the literature as universal rational characters in infinite-dimensional representation theory \cite{Koi89}, the ring itself seems to have received no systematic treatment comparable to the classical theory of symmetric functions as found, for example, in \cite{Mac15, Ful97, Lit06}. The aim of this paper is therefore twofold. First, in analogy to the classical theory, we study the finite variable case of Laurent symmetric polynomials. Second, we define the ring of Laurent symmetric functions, relating it to \cite{Koi89} and \cite{Ful97}.

The organization is as follows. Starting from generalizations of partitions and compositions, we define the ring of \textit{Laurent symmetric polynomials} $\LSym_n$ in \cref{sec:def}. In \cref{sec:ft}, we define the \textit{Laurent monomial symmetric polynomials} $m_\lambda$ for signed partitions of an integer $k$, and show that these form a $\Z$-basis for the Laurent symmetric polynomials of degree $k$. We also show that $\LSym_n$ is generated over $\Z$ by  elementary symmetric polynomials $\{e_1,\dots,e_n,e_{-n}\}$, where $e_{-n}:=(x_1\cdots x_n)^{-1}$ or homogeneous polynomials $\{h_1,\dots,h_n,e_{-n}\}$, and over $\Q$ by power sums $\{p_1,\dots,p_n,e_{-n}\}$. 

In \cref{sec:lscp}, we define the \textit{Laurent alternating polynomials} and \textit{Laurent Schur polynomials}. Then, we extend the Hall inner product to the ring of Laurent symmetric polynomials by defining it to be the constant term of a product (this is the scalar product in \cite[Ch.~VI \S9]{Mac15} for $q=t$), and giving an algebraic proof that this definition satisfies the properties of an inner product. We refer to this inner product as the \textit{Hall inner product}, denoted $\langle \cdot, \cdot \rangle_n$, because it reduces to the usual Hall inner product on  the ring of symmetric polynomials. We give an algebraic proof for this by showing that the Laurent Schur polynomials form an orthonormal basis for the ring of Laurent symmetric polynomials with respect to the Hall inner product. 

In \cref{sec:lsf}, we turn our attention to the case of infinitely many variables. We define the ring of Laurent symmetric functions as a tensor product \[\LSym = \Lambda(X) \otimes_{\Z}\Lambda(Y),\] and show that it can be realized as a direct limit of inverse limits. In \cref{sec:lscf}, we generalize the definition of Schur functions to \textit{Laurent Schur functions}, which we show form a basis for the Laurent symmetric functions. Thus, we extend the definition of the Hall inner product to the ring of Laurent symmetric functions by requiring that the Laurent Schur functions are orthonormal with respect to it. We check that the extended Hall inner product is compatible with the projective map (for large enough $n$, where $n$ is the number of variables) of the aforementioned inverse limit.

In \cref{sec:repGLn}, we discuss an application of Laurent symmetric functions to the representation theory of the general linear group. We begin with a brief review of representations and characters, followed by polynomial and rational representations of $\GL_n$ and their weight space decompositions. Following \cite{Ful97}, we construct the Schur module $\Sc^\lambda(V)$ for a partition $\lambda$ as the image of a Young symmetrizer $y_T$ acting on $V^{\otimes r}$, where $r=\sum_i \lambda_i$. We show that the assignment
\[V \longmapsto \Sc^\lambda(V)\]
is functorial. Hence, each $g \in \mathrm{GL}(V)$ acts on $\Sc^\lambda(V)$ through $\Sc^\lambda(g)$, making $\Sc^\lambda(V)$ a polynomial representation of $\mathrm{GL}(V)$. We then recall equivalent constructions of the Schur module described in \cite{Ful97}, and that the Schur polynomials are the characters of the irreducible polynomial representations of $\mathrm{GL}_n$ (\cref{prop:char-sc}). By tensoring Schur modules with determinant representations, we extend this result to rational representations: the characters of the irreducible rational representations of $\mathrm{GL}_n$ are the Laurent Schur polynomials. We conclude this section by describing the behavior of Schur modules and their characters under taking dual representations.

In \cref{sec:repGLinf}, we turn to the algebraic representation theory of the infinite general linear group \(\mathrm{GL}_\infty\). Letting \(V=\mathbb{C}^{\infty}\) denote the natural representation and \(V_\vee\) its restricted dual, we define an algebraic representation of \(\mathrm{GL}_\infty\) to be a subquotient of a finite direct sum of mixed tensor powers
\[
    V^{\otimes n}\otimes V_\vee^{\otimes m}.
\]
We recall that the irreducible algebraic representations are indexed by signed partitions \((\lambda,\mu)\) \cite[Prop. 3.1.4]{SS15}. Since these representations are generally infinite dimensional, the usual definitions of a character in terms of the trace or weight spaces in $\GL_n$ cannot be used for $\GL_\infty$. We instead define characters using the Grothendieck ring $R(\mathrm{GL}_\infty)$ and results from \cite[\S 3.2]{SS19} and \cite[§2]{Koi89}, which allow us to identify the Laurent Schur functions as the characters of irreducible algebraic representations of $\GL_\infty$, as expected.
\section{Definition\label{sec:def}}

Throughout the text, the natural numbers are \(\N=\{0, 1, 2, \dots\} \). We let $\gS_n$ denote the symmetric group of degree $n$, i.e. the group of permutations of the set $\{1,\dots, n\}$. We denote the set of partitions by \(\Par\), and partitions of length at most \(n\) by \(\Par_n\). We let \(\Pol_n = \Z [x_1,\dots,x_n]\), and denote  the homogeneous polynomials of degree \(k\) by \(\Pol_n^k\). For more background, refer to \cite[Ch.~I]{Mac15}. 

\begin{defin}\label{def:scomp}
    A \emph{signed composition} (or \emph{generalized composition}) $\alpha$ of $k$ is a sequence 
    $(\alpha_{j})_{j = 1}^{\infty} $ of integers, all but finitely many terms of which are zero, such that $\sum_{ j = 1} ^{\infty} \alpha_j = k$.  An \(n\)-tuple of integers \((\alpha_j)_{j = 1}^n\) is identified with a signed composition by appending an infinite number of zeros. 
    
    We define $\SComp(k)$ to be the set of signed compositions of $k$, and $\SComp = \bigcup_{k \in \mathbb{Z}} \SComp(k)$ the set of all signed compositions.  We define
    \begin{equation}
        \SComp_n(k)=\{(\alpha_1, \alpha_2, \dots) \in \SComp(k) \mid \alpha_i=0 \text{ for all } i>n\},
    \end{equation}
    and $\SComp_n = \bigcup_{k \in \mathbb{Z}} \SComp_n(k)$.
\end{defin} 

If $\alpha, \beta \in \SComp_n$, we define $\alpha+\beta \in \SComp_n$ by entry-wise addition. 

We now generalize the notion of a partition in \cite[Ch.~I, \S1]{Mac15} to include negative parts.

\begin{defin} \label{def:SPar}
    A \emph{signed partition} (or \emph{generalized partition}) is an ordered pair 
    \begin{align*}
        \lambda = (\mu, \nu) 
    \end{align*}
    where $\mu $ and $\nu$ are partitions.

    The \emph{size} of $\lambda$ is defined as $\lvert\lambda\rvert = |\mu| - |\nu|$, and the \emph{length} is $\ell(\lambda) = \ell(\mu) + \ell(\nu)$, which is the number of nonzero entries, where \(|\mu|\) and \(\ell(\mu)\) are, respectively, the size and length of a usual partition, cf. \cite[Ch.~I]{Mac15}.  We write $\lambda \vdash_n k$ to indicate that $\lambda$ is a signed partition of $k$ (that is, $|\lambda|=k$) with \(\ell(\lambda) \leq n\). 
    
    We denote the set of all signed partitions by \emph{$\SPar$}, and the signed partitions of size $k \in \Z$ by $\SPar(k)$. The set of all signed partitions of length at most $n$ is denoted by $\SPar_n$ and \(\SPar_n(k)\) denotes the signed partitions of \(\SPar_n\) of size \(k\).
    Note that 
    \begin{align*}
        \SPar = \bigcup_k \SPar(k) = \bigcup_ n \SPar_n.
    \end{align*}
    When considering a fixed \(n\), the elements \(\lambda = (\mu,\nu)\) of \(\SPar_n\) will be written as \(n\)-tuples \((\lambda_j)_{j =1}^n\) of integers consisting of the nonzero entries in \(\mu\), the negative of the nonzero entries of \(\nu\) and zeros in such a way that \(\lambda_1 \geq \lambda_2 \geq \cdots \geq \lambda_n\), cf. the example below.
\end{defin}

\begin{eg}\label{eg:SPAR}
    Consider the signed partition \(\lambda = ((3,1),(4,2,1))\).  Its length is \(\ell (\lambda) = 5\) and \(|\lambda| = -3\). When considered as an element of \(\SPar_7 \) it is written as \((3,1,0,0,-1,-2,-4)\) and when considered as an element of \(\SPar_6\) as \((3,1,0,-1,-2,-4)\).
\end{eg}

\begin{rem} \label{rem:monkey}
    We have $\SPar_n(k)\subseteq \SComp_n(k)$ by the identification of $\lambda = (\mu, \nu) \in \SPar_n$ with the $n$-tuple $(\lambda_1, \dots, \lambda_n)$ and, for every $\alpha \in \SComp_n(k)$, there is a unique $\lambda \in \SPar_n(k)$ such that $\lambda$ is a reordering of $\alpha$. Thus, $\SPar_n(k)$ can be identified with the quotient of $\SComp_n(k)$ under the action of the symmetric group $\gS_n$ (where $\gS_n$ acts by permutation of the entries, see \cref{eq:alpha-pi}). 
\end{rem}

Recall from the theory of symmetric functions that a partition $\lambda$ may be represented by a \emph{Young diagram}, with $\lambda_1$ boxes in the first row, $\lambda_2$ boxes in the second row, and so forth. More formally, we define the Young diagram of \(\lambda\) by the subset of \(\Z^2\) formed by all tuples \((i,j)\) such that \(1 \leq j \leq \lambda_i\). We call each such tuple a \emph{cell}. A Young diagram for a partition $\lambda$ filled with a positive integer in each cell is said to be a \emph{Young tableau} (plural \textit{tableaux}) of \emph{shape} $\lambda$. 
\begin{eg}
    The Young diagram
    \begin{align}
        \ydiagram{4, 2, 1} \label{eq:yd}
    \end{align}
    corresponds to the partition $(4, 2, 1) \in \Par_7$.
\end{eg}

\begin{defin}\label{def:hl}
    For the $(i, j)$-th (row $i$, column $j$) cell in a Young diagram for a partition $\lambda$, the \emph{hook}, denoted $H_\lambda(i, j)$, is the set of cells $(a, b)$ such that $a=i$ and $b\geq j$, or $a \geq i$ and $b =j$. The number of cells in $H_\lambda(i, j)$ is said to be the \emph{hook length} of the $(i, j)$-th cell. 
\end{defin}

\begin{eg}
The following tableau of shape $\lambda=(4, 2, 1)$
\begin{align*}
    \begin{ytableau}
        6 & 4 & 2 & 1 \\
        3 & 1 \\
        1
    \end{ytableau}
\end{align*}
is obtained from the Young diagram \cref{eq:yd} by filling each cell with its hook length. In particular, the $(1, 2)$-th cell has hook length $4$ because $H_\lambda(1,2)$ is given by the following shaded cells.

\begin{align*}
    \begin{ytableau}
    ~&*(Yellow)~ & *(Yellow)~ &*(Yellow)~\\
    ~&*(Yellow)~ & ~ \\
    ~\\
    \end{ytableau}
\end{align*}
\end{eg}

If $\lambda$ is a (usual) partition, we denote by $\lambda'$ the \emph{conjugate} of $\lambda$, which is the partition corresponding to the conjugate of the Young diagram of $\lambda$ (the Young diagram of $\lambda$ reflected across the main diagonal). 
\begin{defin}
    Let $\lambda = (\mu, \nu) \in \SPar$. We define its \emph{conjugate} to be $(\mu', \nu')\in \SPar$, where $\mu'$ is the conjugate of $\mu$ and $\nu'$ is the conjugate of $\nu$. 
\end{defin}

\begin{eg}
    Consider again the signed partition $\lambda = ((3,1),(4,2,1))$ from \cref{eg:SPAR}. Then $\lambda' = (\mu', \nu') = ((2,1,1),(3,2,1,1))$.
\end{eg}

\begin{defin} \label{def:LPol}
We define the \emph{ring of Laurent polynomials (with integer coefficients)} in \(n \) variables to be
\[
\LPol_n := \Z [x_1^{\pm 1},\dots,x_n^{\pm 1}].
\]
For \(\alpha \in \SComp_n\) we let \(x^\alpha := x_1^{\alpha_1}\cdots x_n^{\alpha_n}\). Note that any Laurent polynomial can be written uniquely in the form 
\begin{align}
    \sum_{\alpha \in \SComp_n} c_\alpha x^\alpha, \qquad c_\alpha \in \Z,
\end{align}
with all but finitely many \(c_\alpha\) equal to zero.

We define the subgroup of \emph{Laurent polynomials of degree \(k \in \Z\) in \(n\) variables} to be
\[
\LPol_n^k = \left\{ \sum_{\alpha \in \SComp_n(k)} c_\alpha x^\alpha : c_\alpha \in \Z \right\}
\subseteq \LPol_n,
\]
where the sum over all \(\alpha \in \SComp_n (k)\) is understood to have only finitely many nonzero terms.  In this way, the ring \(\LPol_n\) becomes a \(\Z\)-graded ring 
\[
\LPol_n = \bigoplus_{k \in \Z} \LPol_n^k.
\]
\end{defin}

\begin{eg}
    For $n=2$, we have $x^{(1,0)} + x^{(1,-3)} =x_1+x_1x_2^{-3} \in \LPol_2$, where $x_1=x_1x_2^0 \in \LPol_2^1$ and $x^{(1,-3)} =x_1x_2^{-3} \in \LPol_2^{-2}$.
\end{eg}

Recall that the ring of symmetric polynomials $\Sym_n$ is defined to be the invariant elements of the polynomial ring under the action of the symmetric group \cite[Ch.I, \S2]{Mac15}. Then, let $\mathfrak{S}_n$ denote the symmetric group of order $n!$. 
For $f(x_1, \dots, x_n) \in \LPol_n^k$ and $\pi \in \mathfrak{S}_n$, $\pi$ acts on $f(x_1, \dots, x_n)$ via
\[\pi (f(x_1, x_2, \cdots, x_n) )=f(x_{\pi(1)}, x_{\pi(2)}, \cdots, x_{\pi(n)}).\] 
This defines a left action of $\gS_n$ on $\LPol_n$ for all $n$.

\begin{eg}
    Consider $f(x_1, x_2, x_3)=2x_3x_2^2+x_1^3x_3^2$ and $\pi=(123)$, $\sigma=(12)$, both in cycle notation. By definition, we have $\pi (f(x_1, x_2, x_3))=f(x_{\pi(1)}, x_{\pi(2)}, x_{\pi(3)})=f(x_2, x_3, x_1)$, and 
    \begin{align*}
        \sigma(\pi(f))(x_1, x_2, x_3) =\sigma(f(x_2, x_3, x_1)) =\sigma (2x_1x_3^2+x_2^3x_1^2) = 2x_2x_3^2+x_1^3x_2^2.
    \end{align*}
    On the other hand, $\sigma \circ \pi=(12)(123)=(23)$, so 
    \begin{align}
        (\sigma \circ \pi) f(x_1, x_2, x_3)& =(23) f= 2x_2x_3^2+x_1^3x_2^2.
    \end{align}
\end{eg}

\begin{defin}\label{def:LSymkn}
    Let $\LSym_n^k=(\LPol_n^k)^{\mathfrak{S}_n}=\{f \in \LPol_n^k \mid \pi f=f \text{ for all } \pi \in \mathfrak{S}_n\}$ be the subgroup of Laurent polynomials invariant under the action of $\gS_n$.
    The \emph{ring of Laurent symmetric polynomials in $n$ variables} is 
    \begin{align}\label{eq:LauGrad}
        \LSym_n= (\LPol_n)^{\mathfrak{S}_n} = \bigoplus_{k \in \Z} \LSym_n^k.
    \end{align}
    The second equality comes from the fact that the elements of the symmetric group act as graded ring automorphisms on $\LPol_n$.
\end{defin}

\begin{eg}
     For $n=2$, we have $x_1+x_2+x_1x_2^{-3}+x_1^{-3}x_2 \in \LSym_2$.
\end{eg}

We define a right action of $\mathfrak{S}_n$ on $\SComp_n$. For $\alpha = (\alpha_1, . . . , \alpha_n) \in \SComp_n$ and $\pi \in \mathfrak{S}_n$ we define 
\begin{align}
    \alpha ^\pi := (\alpha_{\pi(1)}, . . .,\alpha_{\pi(n)}). \label{eq:alpha-pi}
\end{align}

The orbit of $\alpha \in\SComp_n(k)$ under the action of $\mathfrak{S}_n$ is denoted by 
\begin{align}
    \alpha\mathfrak{S}_n := \{ \alpha^\pi \mid \pi \in \mathfrak{S}_n\}.
\end{align}

\begin{defin}
    The \emph{monomial symmetric polynomials $m_\lambda$} for $\lambda \in \Par_n$ are
\begin{align*}
    m_\lambda := \sum_{\alpha \in \lambda\gS_n} x^\alpha,
\end{align*}
(for more details, see \cite[Ch.~I, \S2, Eq.~(2.1)]{Mac15}).

We extend this definition to signed partitions. The \emph{Laurent monomial symmetric polynomials $m_\lambda$} are
\begin{align}
    m_\lambda(x_1, \dots, x_n) :=\sum_{\beta \in \lambda \mathfrak{S}_n} x^\beta, \quad \lambda \in \SPar_n(k).
\end{align}
We also denote $m_\lambda(x_1, \dots, x_n ) $ by $m_\lambda^{(n)}$.
\end{defin}

\begin{eg}
    For $\lambda=(2, 0, -1)$, 
    we have \[\lambda\gS_3 = \{(2,-1,0), (2,0,-1), (0,2,-1), (0,-1,2), (-1,0,2),(-1,2,0) \}\] and thus
    \begin{align*}
        m_{\lambda}(x_1, x_2, x_3)=\sum_{\mu \in \lambda \mathfrak{S}_3}x^\mu=x_1^2x_3^{-1}+x_1^2x_{2}^{-1}+x_2^2x_3^{-1}+x_2^{-1}x_3^2+x_{1}^{-1}x_3^2+x_1^{-1}x_2^2.
    \end{align*}
\end{eg}

\begin{prop}
    The Laurent monomial symmetric polynomials $m_\lambda(x_1, \dots, x_n)$ for $\lambda \in \SPar_n(k)$ form a $\Z$-basis for $\LSym_n^k$.
\end{prop}

\begin{proof}
    We show that $\{m_\lambda \mid \lambda \vdash_n k\}$ spans $\LSym_n^k$. Let $f = \sum_\alpha c_\alpha x^\alpha \in \LSym_n^k$. Then $f$ is a sum of monomials of degree $k$. 
    We proceed by induction on the number $q$ of nonzero $c_\alpha$. 
    The case  $q=0$ is clear, since it follows that $f=0$. Suppose the statement holds for all numbers smaller than  $q \in \N$. By definition, $\pi f =f$ for all $\pi \in \mathfrak{S}_n$. Suppose \(c_\alpha x^\alpha\) is a non-zero monomial of \(f\).
    Since $\pi f =f$, every $ x^\beta$, where $\beta \in \alpha\mathfrak{S}_n$ must also appear in $f$ with the same coefficient $c_\alpha$. Thus, if \(\lambda\) is the signed partition obtained from reordering \(\alpha\), then \(f - c_\alpha m_\lambda\) is a Laurent symmetric polynomial with fewer non-zero monomials, and so we may apply the induction hypothesis. Thus, $f - c_\alpha m_\lambda$ is in the span of $\{m_\lambda \mid \lambda \vdash_n k\}$ and so is $f$. 

    It remains to show that $m_\lambda$, $\lambda \in \SPar_n(k)$ are linearly independent. Suppose 
    \begin{align}
        f = \sum_{\lambda \in \SPar_n(k)} c_\lambda m_\lambda =0, \quad c_\lambda \in \Z \text{ for all } \lambda \in \SPar_n(k). \label{eq:ldrel}
    \end{align}
    For \(\lambda, \lambda' \in \SPar_n(k)\) such that \(\lambda \neq \lambda'\) we have \(\lambda\mathfrak{S}_n \cap \lambda' \mathfrak{S}_n = \varnothing\) as a consequence of \cref{rem:monkey}. Thus, the Laurent polynomials \(m_\lambda, m_{\lambda'}\) have no monomials in common so the coefficient of the monomial $x^\lambda $ in $f$ is $c_\lambda$ for all \(\lambda \in \SPar_n (k)\) and thus   \(c_\lambda = 0\)  for $f$ to be equal to zero.
\end{proof}

\section{The fundamental theorem} \label{sec:ft}

In this section, all symmetric polynomials, e.g. $m_\lambda, e_r, h_r, p_r, s_\lambda$, are assumed to be in \(n\) variables, for some natural number $n\geq 1$ fixed throughout this section and the next. Later, in \cref{sec:lsf}, when we consider \textit{symmetric functions} in infinite variables, we will denote the finite variable polynomials with a superscript indicating the number of variables, such as $e_r^{(n)}, s_{\lambda}^{(n)}$, etc.

If \(R\) is a ring and \(S \subseteq R\) a subring, then for \(r_1,\dots,r_n \in R\), \(S[r_1,\dots,r_n] \subseteq R\) is the smallest subring of \(R\) which contains \(S\) and \(r_1,\dots,r_n\). In contrast, we will use uppercase letters, e.g. \(T_1,T_2,\dots\), to denote by $R[T_1, T_2, \dots, T_n]$ the polynomial ring in independent variables.
\begin{defin}
    The \emph{$r$-th elementary symmetric polynomial} in $n$ variables is
     \begin{align*}
         e_r(x_1, \dots, x_n) &= \sum_{1 \leq i_1 < \cdots < i_r \leq n} x_{i_1}\cdots x_{i_r} ,  &\text{for } r &> 0, \\
         e_0(x_1, \dots, x_n)&=1,  &\text{for } r &= 0.
     \end{align*}
     Note that for \( r > n\), \(e_r(x_1,\dots,x_n) = 0\).
\end{defin}

\begin{eg}
    The elementary symmetric polynomials in three variables are 
    \begin{align*}
        e_1(x_1, x_2, x_3)&=x_1+x_2+x_3, \\ 
    e_2(x_1, x_2, x_3)&=x_1x_2+x_2x_3+x_1x_3, \\
    e_3(x_1, x_2, x_3)&=x_1x_2x_3.
    \end{align*}
     In particular, for all $n \in \N_{\geq 1}$, 
     \begin{align}
         e_n(x_1, \dots, x_n)=x_1\cdots x_n.
     \end{align}
\end{eg}

For $\lambda \in \Par$, let $e_\lambda=e_{\lambda_1}\cdots e_{\lambda_{\ell(\lambda)}}$, where $e_i$ are the elementary symmetric polynomials in $n$ variables. 
\begin{prop}[{\cite[Ch.~I, \S2, Eq.~(2.4)]{Mac15}}]\label{prop:ftsp}
For $n \in \N_{\geq 1}$, it holds that
\begin{enumerate}
    \item The elementary symmetric polynomials $e_1(x_1,\dots,x_n),\dots,e_n(x_1,\dots,x_n)$ are algebraically independent over \(\mathbb{Z}\),
    \item The ring of symmetric polynomials in $n$ variables is generated over $\Z$ by the elementary symmetric polynomials,\[\Sym_n=\mathbb{Z}[e_1(x_1,\dots,x_n),\ldots,e_n(x_1,\ldots,x_n)].\]
\end{enumerate}
Thus, the elements $e_\lambda(x_1,\dots,x_n)$ for $\lambda \in \Par,\ \lambda_1 \leq n$ form a $\Z$-basis of $\Sym_n$.
\end{prop}

\cref{prop:ftsp} is often said to be the \textit{fundamental theorem of symmetric polynomials}. 

\begin{theo}\label{theo:gene_n}
\begin{enumerate}
    \item\label{theo3.3a} The elementary symmetric polynomials \(e_1,\dots,e_n\) together with 
    \begin{align}
        e_{-n} := 1/e_n = x_1^{-1}\cdots x_n^{-1}
        \label{eq:e_n=1}
    \end{align}
    generate $\LSym_n$ as a $\Z$-algebra. Equivalently,
    \[\LSym_n=\Z[e_1, \dots, e_n, e_{-n}].\]

    \item\label{theo3.3b} The map given by
    \begin{align*}
        \varphi \colon \frac{\Z[T_1,\dots,T_n,T_{n + 1}]}{(T_n T_{n + 1} - 1)} &\longrightarrow \LSym_n, \\
        T_i &\longmapsto e_i \quad \text{for }i \in \{ 1,\dots,n\}, \\
        T_{n+1} &\longmapsto e_{-n},
    \end{align*}
    is well defined and is an isomorphism.
\end{enumerate}
\end{theo}

\begin{proof}
     For part \cref{theo3.3a},
     we need to show that any Laurent symmetric polynomial can be written as a polynomial in \(e_1,\dots,e_n,e_{-n}\). Indeed, given any \(\ell(x_1,\dots,x_n) \in \mathrm{LSym}_n\), we know that for sufficiently large \(d \in \N\), \(e_n^d \ell(x_1,\dots,x_n) \in \Sym_n\). By \cref{prop:ftsp}, 
     \[ \Sym_n=\Z[e_1,\dots,e_n],\] 
     and hence, there exists a polynomial $q(T_1,\dots,T_n) \in \Z[T_1,\dots,T_n]$ such that \(e_n^d \ell (x_1,\dots,x_n) = q(e_1,\dots,e_n)\).
     Therefore, \(\ell(x_1,\dots,x_n) = e_{-n}^d q(e_1,\dots,e_n)\).

     For part \cref{theo3.3b}, consider the ring morphism
     \[
     \varphi \colon \Z [T_1,\dots,T_n, T_{n + 1}] \longrightarrow \LSym_n
     \]
     defined by \(\varphi(T_i) = e_i\) for \(1 \leq i \leq n\) and \(\varphi (T_{n + 1}) = e_{-n}\), so that \(\varphi(g(T_1,\dots,T_{n},T_{n + 1})) = g(e_1,\dots,e_n, e_{-n})\). From part (a), we know that $\varphi$ is surjective. Since \(T_n T_{n + 1} - 1 \in \ker \varphi\), we can factor through the quotient:
     \begin{equation}
          \begin{tikzcd}[row sep=1cm, column sep=2.5cm,arrow style=tikz]
    \mathbb{Z}[T_1, \dots, T_{n+1}] \arrow[d] \arrow[dr,twoheadrightarrow, "\varphi"] & \\
    \displaystyle\frac{\mathbb{Z}[T_1, \dots, T_{n+1}]}{(T_n T_{n+1} - 1)} \arrow[r, twoheadrightarrow, "\tilde{\varphi}"] & \LSym_n
    \end{tikzcd} \label{eq:phiphitilde}
     \end{equation}
    It remains to show that \(\tilde{\varphi}\) is injective, or equivalently, that if \(\varphi(g) = 0\) for some \(g(T_1,\dots,T_{n + 1}) \in \Z[T_1,\dots,T_{n + 1}]\), then \(g(T_1,\dots,T_{n + 1}) \equiv 0 \pmod{T_n T_{n + 1} - 1}\). Let \(g(T_1,\dots,T_{n + 1}) \in \ker \varphi\). Without loss of generality, suppose
    \[
        g(T_1,\dotsc,T_{n+1}) = \sum_{a,b \in \N} c_{a,b}(T_1,\dotsc,T_{n-1}) T_n^a T_{n+1}^b,
    \]
    where $c_{a,b}(T_1,\dotsc,T_{n-1})\in \Z[T_1,\dotsc,T_{n-1}]$. Since \(T_n T_{n + 1} \equiv 1 \pmod{T_nT_{n + 1} - 1}\) in the quotient $\frac{\mathbb{Z}[T_1, \dots, T_{n+1}]}{(T_n T_{n+1} - 1)}$, 
    \begin{align*}
        g(T_1,\dotsc,T_{n+1}) & = \sum_{a,b \in \N} c_{a,b}(T_1,\dotsc,T_{n-1}) T_n^a T_{n+1}^b \\
        &\equiv \sum_{a \geq b} c_{a,b}(T_1, \dots, T_{n-1})T_n^{a - b} + \sum_{a < b} c_{a,b}(T_1, \dots, T_{n-1})T_{n+1}^{b - a}  \pmod{T_n T_{n + 1} - 1} \\
        &=: g'(T_1,\dots,T_{n+1}) 
    \end{align*}
This last polynomial \(g'(T_1,\dots,T_{n + 1})\) can be written as
 \[
 g'(T_1,\dots,T_{n+1}) = g'_0 (T_1,\dots,T_n) + \sum_{j = 1}^k g'_j(T_1,\dots,T_{n - 1})T_{n +1}^j,
 \]
 for some \(k \in \N\) and some polynomials
 \begin{align*}
     g'_0(T_1,\dots,T_n) &\in \Z[T_1,\dots,T_n] \\
     g'_j(T_1,\dots,T_{n - 1}) &\in \Z [T_1,\dots,T_{n - 1}].
 \end{align*}

    By assumption, $\varphi(g)=0$. Since the diagram \cref{eq:phiphitilde} commutes,
    we have \(\varphi(g') = 0\) as well. As such, we have
    \[
      0 = g'(e_1,\dots,e_n,e_{-n}) = g'_0(e_1,\dots,e_n) + \sum_{i =  1}^{k}g'_i(e_1,\dots,e_{n -1})e_{-n}^i.
    \]
    Multiplying by \(e_n^{k}\) on both sides, we get
    \[
    0=g'_0(e_1,\dots,e_n)e_n^{k} + \sum_{i =  1}^{k}g'_i(e_1,\dots,e_{n -1})e_{n}^{k - i}.
    \]
    This is a polynomial relation of \(e_1,\dots,e_n\). But because \(e_1,\dots,e_n\) are algebraically independent over \(\Z\) (see \cref{prop:ftsp}), the polynomial
    \[
    g'_0(T_1,\dots,T_n)T_n^{k} + \sum_{i =  1}^{k}g'_i(T_1,\dots,T_{n -1})T_{n}^{k - i}
    \]
    must be the 0 polynomial in \(\Z[T_1,\dots,T_n]\). Then, by comparing the coefficients of the variable \(T_{n}\), we find
    \[
    g'_0(T_1,\dots,T_n) = g'_1(T_1,\dots,T_{n - 1}) = \cdots = g'_{k }(T_1,\dots,T_{n - 1}) = 0.
    \]
    Therefore, \(g'(T_1,\dots,T_{n + 1}) = 0\) and \(g(T_1,\dots,T_{n + 1}) \equiv 0 \pmod{T_nT_{n + 1} - 1}\), as desired.
\end{proof}

\begin{rem} We have used in the proof of \cref{theo:gene_n} that $e_{-n}$ defined in \cref{eq:e_n=1} is the inverse of $e_n$. However, in general, it is not true that
\begin{align}
     e_{i}(x_1^{-1}, \dots, x_n^{-1}) = e_i(x_1, \dots, x_n)^{-1}. \label{eq:elt-inv}
\end{align}
In fact, \cref{eq:elt-inv} only holds if $i = n$. 
\end{rem}

\begin{cor} \label{cor:FTLSP}
    The set 
    \begin{align}
        \{e_1^{\alpha_1} \cdots e_{n - 1}^{\alpha_{n - 1}} e_n^{\alpha_n} \mid (\alpha_1,\dots,\alpha_{n -1},\alpha_n) \in \N^{n -1} \times \Z \} \label{eq:elt-basis}
    \end{align} forms a $\Z$-basis of \(\LSym_n\).
\end{cor}
\begin{proof}
        It follows from \cref{theo:gene_n} that \(e_1^{\alpha_1} \cdots e_{n - 1}^{\alpha_{n - 1}} e_n^{\alpha_n}\) spans \(\LSym_n\). It remains to show that the \(e_1^{\alpha_1} \cdots e_{n - 1}^{\alpha_{n - 1}} e_n^{\alpha_n}\) are linearly independent. Assume to the contrary that there is a linear relation 
        \begin{align}
            0&=\sum_{\alpha \in \N^{n -1} \times \Z} c_\alpha e_1^{\alpha_1} \cdots e_{n-1}^{\alpha_{n-1}}e_n^{\alpha_n}, \label{eq:elephant}
        \end{align}
        with the $c_\alpha$ not all zero. As the sum \cref{eq:elephant} is finite, there exists $m = \min\{\alpha_n \mid c_\alpha \neq 0\}$. 
        For this \(m\), we may multiply both sides of \cref{eq:elephant} by $e_n^{-m}$ to get
        \begin{align}
            0&=\sum_{\alpha' \in \N^{n}} c_{\alpha'} e_1^{\alpha'_1} \cdots e_{n-1}^{\alpha'_{n-1}}e_n^{\alpha_n'},
        \end{align}
        where now $\alpha' = (\alpha_1,\dots,\alpha_{n-1}, \alpha_{n} - m) \in \N^n$.
        By \cref{prop:ftsp},  \(e_1^{\alpha_1} \cdots e_{n - 1}^{\alpha_{n - 1}} e_n^{\alpha_n}\), for \((\alpha_1,\dots,\alpha_{n -1},\alpha_n) \in \N^{n}\), are linearly independent, and so $0=c_\alpha$ for every $c_\alpha$ in \cref{eq:elephant}.
        Thus, the only linear relation for \(\{e_1^{\alpha_1} \cdots e_{n - 1}^{\alpha_{n - 1}} e_n^{\alpha_n} \mid (\alpha_1,\dots,\alpha_{n -1},\alpha_n) \in \N^{n -1} \times \Z\}\) is the trivial relation, as desired.
\end{proof}

\begin{defin}
     Let $r, n \in \N$. The \emph{$r$-th complete homogeneous polynomial} in $n$ variables is 
     \begin{align}
         h_r(x_1, \dots, x_n)=\sum_{1\leq i_1 \leq \cdots \leq i_r\leq n} x_{i_1}\cdots x_{i_r}
     \end{align}
     for $r>0$. For $r=0$, 
     \begin{align}
         h_0 = 1.
     \end{align}
\end{defin}
Note that we allow for repeats in the indices $i_1, \dots, i_r$ in the definition of $h_r$, whereas every monomial in $e_r$ is square-free.
\begin{eg}
    For $r,n=2$, we have $h_2(x_1, x_2)=x_1^2+x_2^2+x_1x_2$. 
\end{eg}

\begin{lem}\label{lem:e_nh_n}
    For all $n,r \geq 1$, we have 
    \begin{equation}
        e_r(x_1,\dots,x_n)=
        \det \begin{pmatrix}
        h_1 & h_2 & \cdots & h_r\\
        1 & h_1 & \cdots  & h_{r -1}\\
        \vdots & \ddots & \ddots & \vdots\\
        0 & \cdots & 1 & h_1
        \end{pmatrix}
        ,\qquad h_r (x_1,\dots,x_n) = \det\begin{pmatrix}
        e_1 & e_2 & \cdots & e_r\\
        1 & e_1 & \cdots  & e_{r -1}\\
        \vdots & \ddots & \ddots & \vdots\\
        0 & \cdots & 1 & e_1
        \end{pmatrix}.
    \end{equation}
\end{lem}

\begin{proof}
    This follows from the Jacobi-Trudi identities \cite[Eq.~(3.4)--(3.5)]{Mac15} if one sets $\lambda=(1^r)$ and $\lambda=(r, 0, 0, \dots, 0)$, respectively.
\end{proof}

\begin{cor}\label{cor:omega}
The ring morphism
\begin{align}
    \omega \colon \Z [T_1,\dots,T_n] &\longrightarrow \Z[T_1,\dots,T_n], \\
            T_i &\longmapsto     \det \begin{pmatrix}
        T_1 & T_2 & \cdots & T_i\\
        1 & T_1 & \cdots  & T_{i -1}\\
        \vdots & \ddots & \ddots & \vdots\\
        0 & \cdots & 1 & T_1
    \end{pmatrix}, \label{eq:omega}
\end{align}
satisfies \(\omega^2 := \omega \circ \omega = \mathrm{Id}\).
\end{cor}

\begin{proof}
    Consider the ring isomorphisms
    \begin{align*}
        \varphi \colon \Z[T_1,\dots,T_n] &\longrightarrow \Sym_n = \Z[e_1,\dots,e_n], \\
        T_i &\longmapsto e_i,
    \end{align*}
    and
    \begin{align*}
        \psi \colon \Z[T_1,\dots,T_n] &\longrightarrow \Sym_n = \Z [h_1,\dots,h_n], \\
        T_i &\longmapsto h_i.
    \end{align*}
    We see that
    \[
    (\varphi^{-1} \circ \psi) (T_i) = \varphi^{-1} (h_i) =\varphi^{-1} \left( \det\begin{pmatrix}
        e_1 & e_2 & \cdots & e_i\\
        1 & e_1 & \cdots  & e_{i -1}\\
        \vdots & \ddots & \ddots & \vdots\\
        0 & \cdots & 1 & e_1
        \end{pmatrix}
        \right) = \det \begin{pmatrix}
        T_1 & T_2 & \cdots & T_i\\
        1 & T_1 & \cdots  & T_{i -1}\\
        \vdots & \ddots & \ddots & \vdots\\
        0 & \cdots & 1 & T_1
    \end{pmatrix}.
    \]
    Thus, \(\omega = \varphi^{-1} \circ \psi\). Similarly, one can check that \(\omega = \psi^{-1} \circ \varphi\) and therefore \(\omega^2 = (\varphi^{-1} \circ \psi) \circ (\psi^{-1} \circ \varphi) = \mathrm{Id}\).
\end{proof}

\begin{prop}\label{prop:hgen}
    Let \(n \geq 1\).  Then the Laurent symmetric polynomials \(h_1,\dots,h_n,e_{-n}\) generate the ring of Laurent symmetric polynomials in \(n\) variables as a $\Z$-algebra. 
    Furthermore, if we let 
    \[
    \alpha(T_1,\dots,T_{n +1}) = T_{n+ 1}\det 
    \begin{pmatrix}
        T_1 & T_2 & \cdots & T_n\\
        1 & T_1 & \cdots  & T_{n -1}\\
        \vdots & \ddots & \ddots & \vdots\\
        0 & \cdots & 1 & T_1
    \end{pmatrix} - 1
    \]
     then the morphism $\chi$ described by:
    \begin{align*}
       \chi :\frac{\Z[T_1,\dots,T_{n + 1}]}{(\alpha (T_1,\dots,T_{n+ 1}))}&\longrightarrow \Z[h_1,\dots, h_n, e_{-n}] = \LSym_n, \\
        T_i &\longmapsto h_i \quad \text{for }i \in \{ 1,\dots,n\}, \\
        T_{n+1} &\longmapsto e_{-n},
    \end{align*}
    is well defined and an isomorphism.
\end{prop}

\begin{proof}
As in the proof of \cref{theo:gene_n}, we can multiply any Laurent symmetric polynomial by a large enough power of \(e_n\) to get a symmetric polynomial. Then, since \(\mathrm{Sym}_n = \Z[h_1,\dots,h_n]\) we can deduce that \(\mathrm{LSym}_n = \Z [h_1,\dots,h_n,e_{-n}]\).

For defining the \(\chi\) map, consider the ring morphism
\begin{align*}
    \tau \colon \Z[T_1,\dots,T_n, T_{n + 1}] &\longrightarrow \Z[T_1,\dots,T_n,T_{n + 1}],\\
    T_i &\longmapsto \det      \begin{pmatrix}
        T_1 & T_2 & \cdots & T_i\\
        1 & T_1 & \cdots  & T_{i -1}\\
        \vdots & \ddots & \ddots & \vdots\\
        0 & \cdots & 1 & T_1
    \end{pmatrix}, \qquad 1 \leq i \leq n, \\
    T_{n + 1} &\longmapsto T_{n + 1}.
\end{align*}
Since $\tau$ extends \(\omega\) (as defined in \cref{cor:omega}) to another variable \(T_{n + 1}\) and \(\tau (T_{n + 1}) = T_{n + 1}\), by \cref{cor:omega}, it is easily checked that \(\tau^2 = \mathrm{Id}\), so $\tau$ is an isomorphism. Moreover, since \(\tau (T_n T_{n + 1} - 1) = \alpha (T_1,\dots,T_{n + 1}) \), we have the commutative diagram
\[
\begin{tikzcd}[row sep=large, column sep=large]
\mathbb{Z}[T_1, \dots, T_{n+1}] \arrow[r, "\tau"] \arrow[d] & \mathbb{Z}[T_1, \dots, T_{n+1}] \arrow[d] \\
\dfrac{\mathbb{Z}[T_1, \dots, T_{n+1}]}{(\alpha(T_1, \dots, T_{n+1}))} \arrow[r, "\tilde{\tau}"] & \dfrac{\mathbb{Z}[T_1, \dots, T_{n+1}]}{(T_n T_{n+1} - 1)}
\end{tikzcd}
\]
where \(\tilde{\tau}\) is again an isomorphism. Recall the map \(\varphi \) from \cref{theo:gene_n}. Define \(\chi : = \varphi \circ \tilde{\tau} \colon \Z[T_1,\dots,T_{n + 1}]/(\alpha(T_1,\dots,T_{n + 1})) \rightarrow \LSym_n\). Then \(\chi\) is an isomorphism (since it is a composition of two isomorphisms), given by 
\[
\chi (T_i) = \det\begin{pmatrix}
        e_1 & e_2 & \cdots & e_i\\
        1 & e_1 & \cdots  & e_{i -1}\\
        \vdots & \ddots & \ddots & \vdots\\
        0 & \cdots & 1 & e_1
        \end{pmatrix} = h_i, \qquad 1 \leq i \leq n,
\]
and \(\chi(T_{n + 1}) = e_{-n}\), as desired.
\end{proof}

\begin{defin}
Let $r,n \in \N_{\geq 1}$. The \emph{$r$-th power sum} in $n$ variables is 
\begin{align}
    p_r(x_1, \dots, x_n)=\sum_{i=1}^n x_i^r. 
\end{align}
\end{defin}

\begin{eg} For $n=2$, we have $p_1(x_1, x_2)=x_1+x_2$, $p_2(x_1, x_2)=x_1^2+x_2^2$. 
\end{eg}
Note from the above example that $h_2(x_1, x_2)=\frac{1}{2}p_1(x_1, x_2)^2+\frac{1}{2}p_2(x_1, x_2)$. This motivates us to work over the rational numbers instead of the integers.
\begin{defin}
    The \emph{ring of Laurent polynomials in \(n\) variables with rational coefficients} is 
    \[
    \LPol_{n,\Q} = \Q [x_1^{\pm 1},\dots, x_n^{\pm 1}] \cong \Q \otimes_\Z \LPol_{n}.
    \]
    In analogy to \cref{sec:def}, we define an action of \(\mathfrak{S}_n\) and denote the invariant elements as \(\LSym_{n,\Q}\). There is a canonical isomorphism, \(\LSym_{n,\Q} \cong \Q \otimes_\Z \LSym_n\).
\end{defin}

Our goal is to prove that the \(p_1,\dots,p_n,e_{-n}\) are a set of generators of \(\LSym_{n,\Q}\) and find the relations between them. First, we need some preliminary results.

\begin{lem} \label{lem:endetp}
    For all $n,m \geq 1$, we have
    \begin{gather}
    e_n(x_1,\dots,x_m) = \frac{1}{n!}\det\begin{pmatrix}
        p_1 & 1 & 0 & 0 & \dots & 0\\
        p_2 & p_1 & 2  & 0 & \dots & 0\\
        \vdots & \vdots & \ddots & \ddots & \ddots & \vdots\\
        \vdots & \vdots &  & \ddots & \ddots & 0 \\
        p_{n-1} & p_{n -2} & \dots & \dots & p_1 & n-1 \\
        p_n & p_{n-1} & \dots & \dots & p_2 & p_1 \\
    \end{pmatrix}, \label{eq:er-pr}
    \\
    p_n (x_1,\dots,x_m)= \det\begin{pmatrix}
        e_1 & 1 & 0& 0 & \dots &0\\2e_2 & e_1 & 1 & 0& \dots & 0\\
        \vdots & \vdots & \ddots & \ddots &\ddots & \vdots \\
        \vdots &  \vdots &&\ddots & 1 & 0\\
        \vdots & \vdots &&&e_1 & 1\\
        ne_n & e_{n-1} & e_{n-2} &\dots &\dots& e_1\\
    \end{pmatrix}. \label{eq:pr-er}
    \end{gather}
\end{lem}

\begin{proof}
    The identity \cref{eq:er-pr} is derived in \cite[(6.2; 6), (6.2; 7)]{Lit06} (with $S_r$ being $p_r$ and $a_r$ being $e_r$). The case for \cref{eq:pr-er} is similar.
\end{proof}

\begin{prop}\label{prop:p_ngen}
    Let \(n \geq 1\).  Then the Laurent symmetric polynomials \(p_1,\dots,p_n,e_{-n}\) generate the ring \(\LSym_{n,\Q}:=(\LPol_{n, \Q})^{\mathfrak{S}_n}\). Furthermore, if we let 
    \[
    \beta(T_1,\dots,T_{n +1}) = \frac{1}{n!}T_{n+ 1}\det 
    \begin{pmatrix}
        T_1 & 1 & 0 & 0 & \dots & 0\\
        T_2 & T_1 & 2  & 0 & \dots & 0\\
        \vdots & \vdots & \ddots & \ddots & \ddots & \vdots\\
        \vdots & \vdots &  & \ddots & \ddots & 0 \\
        T_{n-1} & T_{n -2} & \dots & \dots & T_1 & n-1 \\
        T_n & T_{n - 1} & \dots & \dots & T_2 & T_1 \\
    \end{pmatrix} - 1,
    \]
    then the morphism $\phi$ described by
    \begin{align*}
     \phi \colon \frac{\Q[T_1,\dots,T_{n + 1}]}{(\beta (T_1,\dots,T_{n+ 1}))}&\longrightarrow \Q[p_1,\dots, p_n, e_{-n}] = \LSym_{n,\Q}, \\
        T_i &\longmapsto p_i \quad \text{for }i \in \{ 1,\dots,n\}, \\
        T_{n+1} &\longmapsto e_{-n},
    \end{align*}
    is well defined and an isomorphism.
\end{prop}

\begin{proof}
    The fact that \(p_1,\dots,p_n,e_{-n}\) generate \(\LSym_{n,\Q}\) follows from \(\Sym_{n,\Q} = \Q[p_1,\dots,p_n]\) as in the proofs of \cref{theo:gene_n} and \cref{prop:hgen}; we omit the details. For defining \(\phi\), following the proof strategy of \cref{prop:hgen}, we define a ring homomorphism
    \begin{align*}
        \sigma \colon \Q [T_1,\dots,T_n] &\longrightarrow \Q[T_1,\dots,T_n], \\
        T_i &\longmapsto \det\begin{pmatrix}
        T_1 & 1 & 0& 0 & \dots &0\\
        2T_2 & T_1 & 1 & 0& \dots & 0\\
        \vdots & \vdots & \ddots & \ddots &\ddots & \vdots \\
        \vdots &  \vdots &&\ddots & 1 & 0\\
        \vdots & \vdots &&&T_1 & 1\\
        iT_i & T_{i-1} & T_{i-2} &\dots &\dots& T_1\\
    \end{pmatrix},
    \end{align*}
    Consider the isomorphisms
    \begin{align*}
        \eta \colon \Q[T_1,\dots,T_n] &\longrightarrow \Sym_{n,\Q},\\
        T_i  &\longmapsto p_i,
    \end{align*}
    and
    \begin{align*}
        \mu \colon \Q[T_1,\dots,T_n] &\longrightarrow \Sym_{n,\Q},\\
        T_i  &\longmapsto e_i.
    \end{align*}
    We see from \cref{lem:endetp} that \(\sigma = \mu^{-1} \circ \eta\), and so \(\sigma^{-1} = \eta^{-1} \circ \mu\), where we use the relation between $e_i$ and $p_j$ ($j\leq i$) \cref{eq:er-pr}. As such,
    \begin{align*}
        \sigma^{-1} \colon \Q [T_1,\dots,T_n] &\longrightarrow \Q[T_1,\dots,T_n],\\
        T_i &\longmapsto \frac{1}{i!}\det\begin{pmatrix}
        T_1 & 1 & 0 & 0 & \dots & 0\\
        T_2 & T_1 & 2  & 0 & \dots & 0\\
        \vdots & \vdots & \ddots & \ddots & \ddots & \vdots\\
        \vdots & \vdots &  & \ddots & \ddots & 0 \\
        T_{i-1} & T_{i -2} & \dots & \dots & T_1 & i-1 \\
        T_i & T_{i - 1} & \dots & \dots & T_2 & T_1 \\
    \end{pmatrix}. 
    \end{align*}

    If we extend it to \(\Q[T_1,\dots,T_{n + 1}]\) leaving \(T_{n + 1}\) fixed,
     \begin{align*}
        \sigma \colon \Q [T_1,\dots,T_{n + 1}] &\longrightarrow \Q[T_1,\dots,T_{n + 1}],\\
        T_i &\longmapsto \det\begin{pmatrix}
        T_1 & 1 & 0& 0 & \dots &0\\
        2T_2 & T_1 & 1 & 0& \dots & 0\\
        \vdots & \vdots & \ddots & \ddots &\ddots & \vdots \\
        \vdots &  \vdots &&\ddots & 1 & 0\\
        \vdots & \vdots &&&T_1 & 1\\
        iT_i & T_{i-1} & T_{i-2} &\dots &\dots& T_1\\
    \end{pmatrix}, \quad 1 \leq i \leq n,\\
    T_{n + 1} &\longmapsto T_{n + 1},
    \end{align*}
    then the inverse is given by
    \begin{align*}
        \sigma^{-1} \colon \Q [T_1,\dots,T_{n + 1}] &\longrightarrow \Q[T_1,\dots,T_{n + 1}],\\
        T_i &\longmapsto \frac{1}{i!}\det\begin{pmatrix}
        T_1 & 1 & 0 & 0 & \dots & 0\\
        T_2 & T_1 & 2  & 0 & \dots & 0\\
        \vdots & \vdots & \ddots & \ddots & \ddots & \vdots\\
        \vdots & \vdots &  & \ddots & \ddots & 0 \\
        T_{i-1} & T_{i -2} & \dots & \dots & T_1 & i-1 \\
        T_i & T_{i - 1} & \dots & \dots & T_2 & T_1 \\
    \end{pmatrix}, \quad 1\leq i \leq n, \\
    T_{n + 1} &\longmapsto T_{n + 1}.
    \end{align*}
    Since \(\sigma^{-1} (T_n T_{n + 1} - 1) = \beta (T_1,\dots,T_{n + 1})\), we have a commutative diagram
    \[
    \begin{tikzcd}[row sep=large, column sep=large]
    \mathbb{Q}[T_1, \dots, T_{n+1}] \arrow[r, "\sigma"] \arrow[d] & \mathbb{Q}[T_1, \dots, T_{n+1}] \arrow[d] \\
    \dfrac{\mathbb{Q}[T_1, \dots, T_{n+1}]}{(\beta(T_1, \dots, T_{n+1}))} \arrow[r, "\tilde{\sigma}"] & \dfrac{\mathbb{Q}[T_1, \dots, T_{n+1}]}{(T_n T_{n+1} - 1)}
    \end{tikzcd}
    \]
    where \(\tilde{\sigma}\) is again an isomorphism. Let \(\varphi_\Q\) be the scalar extension of the ring isomorphism $\varphi$ from \cref{theo:gene_n} (hence $\varphi_\Q$ is itself an isomorphism), 
    \begin{align*}
        \varphi_\Q \colon \frac{\Q[T_1,\dots,T_n,T_{n + 1}]}{(T_n T_{n + 1} - 1)} &\longrightarrow \LSym_{n, \Q}, \\
        T_i &\longmapsto e_i \quad \text{for }i \in \{ 1,\dots,n\}, \\
        T_{n+1} &\longmapsto e_{-n}.
    \end{align*}
    Then \(\phi = \varphi_\Q \circ \tilde{\sigma}\) is an isomorphism.
\end{proof}

\section{Laurent Schur polynomials} \label{sec:lscp}
\begin{defin}\label{def:LAlt_n}
We define the $\Z$-module of all \emph{Laurent alternating polynomials} or \emph{Laurent skew-symmetric polynomials}  in $n$ variables to be
\[
\LAlt_n := \{ f \in \LPol_n\,|\, \pi(f) = \sgn(\pi)f \text{ for all } \pi \in \mathfrak{S}_n\},
\]
where \(\sgn(\pi) = \pm 1\) denotes the sign of the permutation \(\pi\). Similarly, let $\Alt_n := \{ f \in \Pol_n\,|\, \pi(f) = \sgn(\pi)f \text{ for all } \pi \in \mathfrak{S}_n\}$ in the non-Laurent case.
\end{defin}

We construct alternating polynomials by antisymmetrizing monomials. Let $\alpha \in \SComp_n$ and recall that $x^\alpha = x_1^{\alpha_1}\cdots x_n^{\alpha_n}$. We define
\begin{equation}\label{eq:aalpha}
      a_\alpha = a_\alpha(x_1,\dots, x_n) := \sum_{\pi \in\mathfrak{S}_n} \sgn(\pi) x^{(\alpha^\pi)},
\end{equation}
where $\alpha^\pi$ is defined by \cref{eq:alpha-pi}.  Then $a_\alpha$ is an alternating polynomial.

\begin{rem} \label{rem:1aalpha}
    Note that, by construction, if $\alpha$ has two equal entries, then $a_\alpha = 0$. This implies that there are no repeated terms in the sum \cref{eq:aalpha}, so for any $\alpha \in \SComp_n$, with $a_\alpha \neq 0$, there exists a unique $\pi \in \mathfrak{S}_n$ such that $\alpha^\pi \in \SPar_n$. Then $a_\alpha = \sgn(\pi)a_{\alpha^\pi}$, which motivates us to restrict our attention to $a_\lambda$, where $\lambda \in \SPar_n$.
\end{rem}
    Note that $\LAlt_n$ is not a subring of $\LPol_n$. However, if $f , g  \in \LAlt_n$, then
\[
\pi(fg) = (\pi(f))(\pi(g))=\sgn(\pi)f\sgn(\pi)g = fg \text{ for all } \pi \in \mathfrak{S}_n,
\]
and so $fg \in \LSym_n$. Moreover, if $f \in \LSym_n$ and $g \in \LAlt_n$, then
\[
\pi(fg) = \pi(f) \pi(g)=f\sgn(\pi)g = \sgn(\pi)fg \text{ for all } \pi \in \mathfrak{S}_n,
\]
and so $fg \in \LAlt_n$. It follows that $\LAlt_n$ is an $\LSym_n$-module, as it is closed under multiplication by any element of $\LSym_n$.

Let $\delta^n:=(n-1, n-2, \dots, 0) \in \SPar_n$. 

\begin{lem}\label{lem:lamdel}
    We have a bijection of sets 
\begin{align*}
        \varphi\colon \SPar_n &\longrightarrow \{ \lambda \in \SPar_n \mid \lambda_1 > \cdots > \lambda_n \} =:M. \\
        \lambda & \longmapsto \lambda + \delta^n.
\end{align*}
\end{lem}

\begin{proof}
    We show that the map is well defined, then define an inverse and prove that it maps into $\SPar_n$ proving bijectivity.

    For $\lambda \in \SPar_n$ it holds that $\lambda_1 \geq \ldots \geq \lambda_n$ and thus $\lambda + \delta^n \in M$. 
    For $\gamma \in M$, we define $\psi(\gamma) = \gamma - \delta^n $. This defines a map from $M $ to $\SPar_n$ since for $\gamma \in M $, $\gamma_1 > \ldots > \gamma_n$ and thus, $\gamma_1 - (n-1) \geq \gamma_{2} - (n-2) \geq \ldots \geq \gamma_{n-1} - 1 \geq \gamma_n - 0  $,
    so we have $\gamma - \delta^n \in \SPar_n$. 
    As $\varphi$ and $\psi $ are clearly inverses to each other, bijectivity follows.
\end{proof}
\begin{lem}\label{lem:altdet}
    For $\alpha \in \SComp_n$, it holds that 
    \begin{equation}\label{eq:avocado}
        a_{\alpha} = \det(x_i^{\alpha_j})_{1\leq i,j\leq n} = \det \begin{pmatrix}
            x_1^{\alpha_1} & x_1^{\alpha_2} & \dots &x_1^{\alpha_n}\\x_2^{\alpha_1} & x_2^{\alpha_2} & \dots & x_2^{\alpha_n}\\
            \vdots &\vdots & \ddots & \vdots\\
            x_n^{\alpha_1} & x_n^{\alpha_2} & \dots & x_n^{\alpha_n}
        \end{pmatrix}.   \end{equation} 
\end{lem}
\begin{proof}
By the Leibniz formula, we have
\[\det\left(x_i^{\alpha_j}\right)_{1\leq i,j\leq n}=\sum_{\sigma\in S_n}\sgn(\sigma)\prod_{j=1}^nx_{\sigma(j)}^{\alpha_j},\]
which is the same as $a_\alpha$ defined by \cref{eq:aalpha}. 
\end{proof}

\begin{cor}\label{cor:ea}
    For any $\alpha \in \SPar_n$, and $k \in \Z$ we have
    \begin{align}\label{eq:toast}
        a_{\alpha}(x_1,\dots, x_n) = e_{-n}^ka_{\alpha + (k^n)}(x_1,\dots,x_n).
    \end{align}
\end{cor}

\begin{rem}\label{rem:LS-nonLS}
Note that, because each $\lambda \in \SPar_n$ has only finitely many parts unequal to zero, there is a lower bound on the parts of $\lambda$. Thus, there exists a $k \in \N$ such that for all $r\in \N$ with $r\geq k$, $\lambda+(r^n) \in \Par_n$, and thus $a_{\lambda + (r^n)} \in \mathrm{Alt}_n$.
\end{rem}

\begin{lem}\label{lem:adelta}
The polynomial
\begin{equation}
    a_{\delta^n}(x_1,\ldots,x_n)=\det\left(x_i^{n-j}\right)_{1\leq i,j\le n}=\prod_{1\leq i<j\leq n}(x_i-x_j) 
\end{equation} 
divides every element of $\Alt_n$.
\end{lem}

\begin{proof}
    Let $f \in \Alt_n$ be arbitrary. Since $f$ evaluates to $0$ when $x_i=x_j$, it is divisible by $x_i-x_j$. Since the $x_i-x_j$ are pairwise relatively prime and $\Z[x_1, \dots, x_n]$ is a unique factorization domain, their product $\prod_{1\le i<j\le n}(x_i-x_j)$ divides $f$ in $\Z[x_1, \dots, x_n]$. 
\end{proof}

\begin{prop}{\cite[Ch.~I, \S3, p.~40]{Mac15}}\label{prop:alt-basis}
     The map $\Sym_n \to \Alt_n$, $f \mapsto f\,a_{\delta^n}(x_1,\ldots,x_n)$ is an isomorphism of $\Sym_n$-modules. Thus, the elements 
    \begin{align} 
        a_{\lambda + \delta^n}(x_1, \ldots, x_n) , ~~ \lambda \in \Par_n,
    \end{align}
    form a $\Z$-basis of $\Alt_n$.
\end{prop}

We will now use this to prove the analogous statement for the Laurent alternating polynomials.

\begin{prop} \label{prop:basis-LAlt}
    The elements 
    \begin{align} \label{eq:shifted-alt}
        a_{\lambda + \delta^n}(x_1, \ldots, x_n) , ~~ \lambda \in \SPar_n,
    \end{align}
    form a $\Z$-basis of $\LAlt_n$.
\end{prop}

\begin{proof}
    Let $g \in \LAlt_n$. There is some $k$ large enough such that $e_{n}^kg \in \mathrm{Alt}_n$ ($e_{n}^k$ is symmetric so $e_{n}^kg$ is alternating). By \cref{prop:alt-basis}, there exist $c_\lambda \in \Z$ such that $e_{n}^k g=\sum_{\lambda \in \mathrm{Par}_n} c_\lambda a_{\lambda+\delta^n}$  . Then \begin{align*}
        g&=e_{-n}^k\sum_{\lambda \in \mathrm{Par}_n} c_\lambda a_{\lambda+\delta^n} \underset{\cref{eq:toast}}{=}\sum_{\lambda \in \mathrm{Par}_n} c_\lambda a_{\lambda-(k^n)+\delta^n} = \sum_{\lambda ' \in \mathrm{SPar}_n} c_{\lambda '} a_{\lambda ' +\delta^n},
    \end{align*}
    as desired.

    We have shown that the $a_{\lambda + \delta^n}$ generate $\LAlt_n$. To see that they are linearly independent, suppose there exists a finite linear combination equal to zero,
    \begin{align*}
        0 = \sum_{\lambda \in \SPar_n} c_\lambda a_{\lambda+\delta^n}(x_1, \dots , x_n ),
    \end{align*}
    with $c_\lambda \in \Z$. Then there exists a $k$ such that $\lambda + (k^n) \in \Par_n$ for all $\lambda$ with $c_\lambda \neq 0$ (using \cref{rem:LS-nonLS} and taking the finite maximum of the \(k\) needed for each \(c_\lambda \neq 0\)). Thus, 
    \begin{align*}
        0 &= \sum_{\lambda \in \SPar_n} c_\lambda e_n ^k a_{\lambda + \delta^n}(x_1, \ldots ,x_n)  \\
        &= \sum_{\lambda \in \SPar_n } c_\lambda a_{\lambda + (k^n) + \delta^n} (x_1, \ldots ,x_n) \\
        &= \sum_{\mu \in \Par_n} c_{\mu - (k^n)} a_{\mu + \delta^n}(x_1, \ldots , x_n).
    \end{align*}
    Thus, we get that all $c_\lambda = 0$ because the $a_{\lambda+\delta^n}$ for $\lambda \in \Par_n$ are linearly independent by \cref{prop:alt-basis}. 
\end{proof}

\begin{rem} \label{rem:integral}
    Note that the ring \(\LSym_n\) is a subring of \(\LPol_n\) which is a subring of the quotient field of polynomials in \(n\) variables. Thus, \(\LSym_n\) is an integral domain.
\end{rem}

\begin{lem}\label{lem:vandermond}
    For any \(t(x_1,\dots,x_n) \in \LAlt_n\), there exists a unique element \(\ell(x_1,\dots,x_n) \in \LSym_n\) such that
    \begin{align}
        t(x_1,\dots,x_n) =\ell(x_1,\dots,x_n)a_{\delta^n}(x_1,\dots,x_n).
    \end{align}
\end{lem}

\begin{proof}
    By multiplying by a large enough power \(e_n^k\), we have that \(e_n^k t(x_1,\dots,x_n) \in \mathrm{Alt}_n\). By \cref{prop:alt-basis}, there exists a polynomial \(q(x_1,\dots,x_n) \in \Sym_n\) such that
    \[
    e_n^k t(x_1,\dots,x_n) = q(x_1,\dots,x_n) a_{\delta^n}(x_1,\dots,x_n).
    \]
    Therefore, it suffices to take \(\ell(x_1,\dots,x_n) = e_{-n}^k q(x_1,\dots,x_n)\), which is symmetric because both $e_{-n}^k(x_1,\dots,x_n)$ and $q(x_1,\dots,x_n)$ are symmetric. This $\ell(x_1, \dots, x_n)$ is unique because $\LPol_n$ is an integral domain. 
\end{proof}

\begin{defin}\label{def:schur}
    For $\lambda \in \SPar_n(k)$, the \emph{Laurent Schur polynomials} in $n$ variables are defined as
    \begin{align}
        s_\lambda(x_1, \ldots, x_n) := \dfrac{a_{\lambda + \delta^n}(x_1, \ldots, x_n)}{a_{\delta^n}(x_1, \ldots, x_n)} = \frac{\det(x_i^{\lambda_j + n - j})_{1\leq i,j\leq n}}{\det(x_i^{n - j})_{1\leq i,j\leq n}} \in \LSym_n^{k}. \label{eq:Laurent-Schur}
    \end{align}
    In particular, if $\lambda \in \Par_n$, then $s_\lambda$ is the usual Schur polynomial.
\end{defin}

\begin{rem}
    Note that the second equality in \cref{def:schur} is given by \cref{lem:altdet}. The fact that $s_\lambda(x_1, \dots, x_n)$ is a  uniquely determined Laurent symmetric polynomial is due to \cref{lem:vandermond} and \(\LSym_n\) being an integral domain (\cref{rem:integral}).
\end{rem}

\begin{prop} \label{prop:schuren}
    For all $k \in \Z,\, \lambda \in  \SPar_n$, we have
    \begin{align}
        s_\lambda(x_1,\dots,x_n)  = e_{-n}^ks_{\lambda + (k^n)}(x_1,\dots,x_n). \label{eq:schuren}
    \end{align}
\end{prop}

\begin{proof}
From \cref{cor:ea}, we have 
\begin{align*}
    a_{\lambda + \delta^n}(x_1,\dots,x_n) = e_{-n}^ka_{\lambda + (k^n) + \delta^n}(x_1,\dots ,x_n).
\end{align*}
By \cref{lem:vandermond}, we may divide both sides by $a_{\delta^n}(x_1,\dots,x_n)$, and the claim follows.
\end{proof}

\begin{theo}
    The Laurent Schur polynomials $s_\lambda$ for $\lambda \in \SPar_n$ form a basis of $\LSym_n$ over $\Z$. \label{theo:Schur-basis}
\end{theo}

\begin{proof}
    For each $\ell \in \LSym_n$ there exists a $k \in \N$ such that $g = e_n^{k} \ell \in \Sym_n$. As the usual Schur polynomials form a basis for $\Sym_n$ \cite[Ch.~I, \S~3, Eq.~(3.2)]{Mac15}, we may express $g$ as a linear combination of Schur polynomials:
    \begin{align*}
        g = \sum_{\lambda \in \Par_n} c_\lambda s_\lambda(x_1, \ldots , x_n),\qquad
        c_\lambda \in \Z.
    \end{align*}
    Thus,
    \begin{align*}
        \ell &= e_n^{-k}  \sum_{\lambda \in \Par_n} c_\lambda s_\lambda(x_1, \ldots , x_n) \\
        &=  \sum_{\lambda \in \Par_n} c_\lambda  e_n^{-k} s_\lambda (x_1, \ldots , x_n) \\
        &= \sum_{\lambda \in \Par_n} c_\lambda s_{\lambda - (k^n)} (x_1, \ldots, x_n) \\
        &= \sum_{\mu \in \SPar_n} c_{\mu+(k^n)} s_{\mu} (x_1, \ldots, x_n),
    \end{align*}
    where we set $c_\mu = 0 $ for all $\mu \notin \Par_n$. We have thus shown that $\ell$ is in the span of the $s_\lambda$ and thus, they span $\LSym_n$.
    To show that they are linearly independent, suppose there exists a finite linear combination equal to zero:
    \begin{align*}
        0 = \sum_{\lambda \in \SPar_n} c_\lambda s_\lambda(x_1, \ldots ,x_n),\qquad
        c_\lambda \in \Z.
    \end{align*}
    Then there exists a $k$ such that $\lambda + (k^n) \in \Par_n$ for all $\lambda$ with $c_\lambda \neq 0$ (using \cref{rem:LS-nonLS} and taking the finite maximum of the \(k\) needed for each \(c_\lambda \neq 0\)). Thus, 
    \begin{align*}
        0 &= \sum_{\lambda \in \SPar_n} c_\lambda e_n ^k s_\lambda(x_1, \ldots ,x_n)  \\
        &= \sum_{\lambda \in \SPar_n } c_\lambda s_{\lambda + (k^n)} (x_1, \ldots ,x_n) \\
        &= \sum_{\mu \in \Par_n} c_{\mu - (k^n)} s_\mu(x_1, \ldots , x_n),
    \end{align*}
    and so we get that all $c_\lambda = 0$ because the $s_\lambda$ for $\lambda \in \Par_n$ are linearly independent.
\end{proof} 

\begin{cor}
    $\LAlt_n$ is a free $\LSym_n$-module of rank one. 
\end{cor}

\begin{proof}
    We show that $\LAlt_n$ is isomorphic to $\LSym_n$, as an $\LSym_n$-module, under the following isomorphism:
    \begin{align*}
        \varphi \colon \LSym_n &\rightarrow \LAlt_n, \\
        \ell &\longmapsto a_{\delta^n}\ell.
    \end{align*}
    The map $\varphi$ sends $s_\lambda$ to $a_{\lambda+\delta^n}$.
    Since the former form a $\Z$-basis of $\LSym_n$ and the latter form a $\Z$-basis of $\LAlt_n$, the statement follows. 
\end{proof}

Recall that $e_{-r}(x_1, \dots, x_n)=e_r(x_1^{-1}, \dots, x_n^{-1})$. Likewise, let $h_{-r}(x_1, \dots, x_n)=h_r(x_1^{-1}, \dots, x_n^{-1})$. 
\begin{prop}\label{prop:schur_det}
    Let $\lambda \in \SPar_n$, and \(r,d,m \in \N\) such that \(\lambda + (d^n) \in \Par_n\), $r \geq \ell(\lambda + (d^n))$ and $m\geq \ell((\lambda + (d^n))')$. Then
    \begin{align*}
        s_{\lambda} &= e_{-n}^d \det(\mathds{1}_{\N}((\lambda + (d^n))_i - i + j)h_{(\lambda + (d^n))_i - i + j})_{1\leq i,j\leq r},
        \qquad \text{and}
    \\
         s_{\lambda} &= e_{-n}^d \det(\mathds{1}_{\N}((\lambda + (d^n))'_i - i + j)e_{(\lambda + (d^n))'_i - i + j})_{1\leq i,j\leq m},
         \end{align*}
    where $\mathds{1}_\N$ is the indicator function, defined as
    \begin{align*}
        \mathds{1}_\N(x) := \begin{cases}
            1, & \text{if } x \in \N\\ 0,& \text{if } x \notin \N.
        \end{cases}
    \end{align*}
\end{prop}
\begin{rem} We have used the indicator function $\mathds{1}_\N$ to have 0 entries where the subscript $r$ of $e_r,h_r$ are negative. This is to enforce the non-Laurent condition supposed in the original Jacobi-Trudi formula, i.e, $e_r=h_r=0$ for $r<0$ in \cite[Ch.~I, \S 3, Eqs.3.4-3.5]{Mac15}.
\end{rem}

\begin{proof}
    We only prove the first equality, as the second one is analogous. Indeed, for any $\lambda \in \SPar_n$,
    \begin{align*}
        s_{\lambda}(x_1,\dots,x_n) &=e_{-n}^ds_{\lambda + (d^n)}(x_1,\dots,x_n) \\&= e_{-n}^d\det(\mathds{1}_\N((\lambda + (d^n))_i - i + j)h_{(\lambda + (d^n))_i - i + j}),
    \end{align*}
where we have used in the first equality \cref{prop:schuren}, and the second one follows from the Jacobi-Trudi identities for finite variables (see \cite[Ch.~I, \S 3, Eqs.3.4-3.5]{Mac15}).
\end{proof}

Cauchy's formula in \cite[Ch.~I, Eq.~(4.3)]{Mac15}
can be generalized as follows.
\begin{prop}
     Let $k \in \Z$. We have
     \begin{align}
         e_n^k(x)e_n^k(y)\prod_{i,j =1}^n(1 - x_iy_j)^{-1} = \sum_{\mu \in \SPar_n^{\geq k}}s_\mu(x)s_\mu(y),
     \end{align}
     where $\SPar_n^{\geq k}:= \{\mu \in \SPar_n \mid \mu_n \geq k\}$. 
\end{prop}

\begin{proof}
    From the proof of \cite[Ch.~I, \S3, Eq.~(4.3)]{Mac15}
    we have
    \begin{align*}
        \prod_{i,j =1}^n(1 - x_iy_j)^{-1} &= \sum_{\mu \in \Par_n}s_\mu(x)s_\mu(y) \\&= e_{-n}^k(x) e_{-n}^k(y)\sum_{\mu \in \Par_n}s_{\mu {+} (k^n)}(x)s_{\mu {+} (k^n)}(y)
        \\ &= e_{-n}^k(x) e_{-n}^k(y)\sum_{\nu \in \SPar_n^{\geq k}}s_{\nu}(x)s_{\nu}(y),
    \end{align*}
    where the second equality is given by \cref{prop:schuren}, and the third equality holds since the map 
    \begin{align*}
        \rho \colon \Par_n &\longrightarrow\SPar_n^{\geq k},\\
        \mu &\longmapsto \mu+ (k^n),
    \end{align*}
    is bijective. 
    Then, by multiplying by $e_n^k(x)e_n^k(y)$ on both sides, the claim follows.
\end{proof}

    \begin{defin}\label{def:hall}
        The 
        \emph{Hall inner product}
        of two Laurent symmetric polynomials $f,g \in \LSym_n$ is
        \begin{align}
            \langle f, g \rangle_n
            = \frac{1}{n!} \operatorname{CT} \left( f(x_1, \dots , x_n) g(x_1^{-1}, \dots, x_n^{-1}) \prod_{i \ne j} \left(1 - \frac{x_i}{x_j}\right) \right), \label{eq:Hall-IP}
        \end{align}
        where $\operatorname{CT}$ denotes the constant term of the Laurent polynomial.
    \end{defin}
 
    Evidently, $\cref{eq:Hall-IP}$ is a bilinear map \(\LSym_n \times \LSym_n \to \Q\). Later we will see that the pairing \cref{eq:Hall-IP} is indeed an inner product (i.e., symmetric and positive definite) on $\LSym_n$. 

    \begin{prop} \label{prop:bilin-symm}
    The pairing \cref{eq:Hall-IP} is symmetric. 
    \end{prop}
    \begin{proof}
        Let $f(x)$ and $f(x^{-1})$ be shorthands for $f(x_1, \dots , x_n)$ and $ f(x_1^{-1}, \dots ,x_n^{-1}) $ respectively. For all $h \in \Z[x^{\pm 1}_1, \dots , x_n^{\pm}]$ we have $\operatorname{CT}(h(x)) = \operatorname{CT}(h(x^{-1})) $ and thus 
        \begin{align*}
            \langle f, g \rangle_n
            &= \frac{1}{n!} \operatorname{CT} \left( f(x_1,\dotsc,x_n) g(x_1^{-1},\dotsc,x_n^{-1}) \prod_{i \ne j} \left( 1 - \frac{x_i}{x_j} \right) \right) \\
            &= \frac{1}{n!} \operatorname{CT} \left( f(x_1^{-1},\dotsc,x_n^{-1}) g(x_1,\dotsc,x_n) \prod_{i \ne j} \left( 1 - \frac{x_j}{x_i} \right) \right) \\
            &= \langle g, f \rangle_n,
        \end{align*}
        since the product $\prod_{i \ne j} \left(1 - \frac{x_i}{x_j}\right)$ is invariant under exchanging $i$ and $j$.
        \end{proof}
    
     \begin{lem}\label{lem:tech}
    Fix $\lambda \in \SPar_n$. Then 
         \begin{equation}\label{eq:tild_lamb}
             a_{\lambda}(x_1^{-1}, \dots, x_n^{-1}) = (-1)^{\frac{n(n-1)}{2}}a_{\lambda^\vee}(x_1,\dots,x_n),
         \end{equation}
         where $\lambda^\vee := ( - \lambda_n, \dots, -\lambda_1) \in \SPar_n$.
     \end{lem}
    
     \begin{proof}
         It follows from \cref{eq:avocado} that
         \begin{align*}
             a_{\lambda}(x_1^{-1}, \dots, x_n^{-1}) = \det(x_i^{-\lambda_j})_{1\leq i,j\leq n}.
         \end{align*}
         Since
         \begin{align*}
        \det\begin{pmatrix}
            x_1^{-\lambda_1} & x_1^{-\lambda_2} & \dots & x_1^{-\lambda_n} \\
            \vdots & \vdots &   & \vdots\\
            \vdots & \vdots &   & \vdots\\
            x_n^{-\lambda_1} & x_n^{-\lambda_2}&\dots &x_n^{-\lambda_n}
             \end{pmatrix}  = (-1)^{\frac{n(n-1)}{2}}  \det\begin{pmatrix}
                 x_1^{-\lambda_n} &\dots & x_1^{-\lambda_2} & x_1^{-\lambda_1} &\\
                \vdots && \vdots   & \vdots\\
                \vdots && \vdots    & \vdots\\
                x_n^{-\lambda_n} &\dots& x_n^{-\lambda_2}&x_n^{-\lambda_1}
             \end{pmatrix},
        \end{align*}
         the result follows.   
    \end{proof}
    
     \begin{cor}
         In particular, we have 
         \begin{align}
             a_{\delta^n}(x_1^{-1},\dots, x_n^{-1}) = (-1)^{\frac{n(n-1)}{2}}e_{-n}^{n-1}a_{\delta^n}(x_1,\dots,x_n)
         \end{align}
     \end{cor}
     \begin{proof}
     From \cref{lem:tech}, we have
        \begin{align*}
            a_{\delta^n}(x_1^{-1},\dots, x_n^{-1}) &= (-1) ^{\frac{n(n-1)}{2}} a_{\left(\delta^n\right)^\vee}(x_1,\dots,x_n) \\
            & = (-1) ^{\frac{n(n-1)}{2}} e_{-n}^{n-1} a_{((n - 1)^n) + \left(\delta^n\right)^\vee}(x_1, \dots ,x_n),
        \end{align*} 
    where the last equality follows from \cref{cor:ea}. Since $((n-1)^n)  + \left(\delta^n\right)^\vee = \delta^n$, the claim follows.
     \end{proof}
    
    \begin{prop}\label{prop:relsch}
        Let $\lambda \in \SPar_n$. Then we have
        \begin{align}
            s_\lambda(x_1^{-1}, \dots, x_n^{-1}) = s_{\lambda^\vee}(x_1,\dots,x_n),
        \end{align}
        where $\lambda^\vee$ is as defined in \cref{lem:tech}.
    \end{prop}
    \begin{proof}
        In light of \cref{eq:tild_lamb}, we have
        \[
        s_\lambda (x^{-1}) = \frac{a_{\lambda + \delta^n} (x^{-1})}{a_{\delta^n}(x^{-1})} = \frac{a_{\lambda^\vee + \left(\delta^n\right)^\vee}(x)}{a_{\left(\delta^n\right)^\vee}(x)} \overset{(\star)}{=} \frac{e_{-n}^{n - 1}a_{\lambda^\vee + \delta^n}(x)}{e_{-n}^{n -1}a_{\delta^n}(x)}  = s_{\lambda^\vee}(x),
        \]     
        where $(\star)$ follows from \cref{cor:ea}.
    \end{proof}
     \begin{prop}\label{prop:SchurOGB}
         Let $\lambda, \mu \in \SPar_n$. Then the Laurent Schur polynomials $s_{\lambda}, s_{\mu}$ are orthonormal with respect to the pairing \cref{eq:Hall-IP},
         \begin{align}
             \langle s_\lambda, s_\mu \rangle_n=\delta_{\lambda, \mu}, \label{eq:SchurOGB}
         \end{align}
         where \(\delta_{\lambda, \mu}\) is the Kronecker delta, which is equal to \(1\) if \(\lambda = \mu\) and 0 otherwise.
     \end{prop}
     \begin{proof}
    Recall the Hall inner product from \cref{def:hall}. It suffices to prove that 
    \begin{align}
       \operatorname{CT}\left(s_{\lambda}(x)s_{\mu}(x^{-1})\prod_{i\neq j}(1 -  x_i/x_j)\right) = n!\delta_{\lambda,\mu}
    \end{align}
    The product $\prod_{i\neq j}(1 - x_i/x_j)$ can be written as the product of $P_{i,j}$, for all \(i < j\), where
    \begin{equation*}
        P_{i,j} := \left(1 - \frac{x_i}{x_j}\right)\left(1 - \frac{x_j}{x_i}\right) = \left(\frac{x_j - x_i}{x_j}\right)\left(\frac{x_i - x_j}{x_i}\right) = \left(\frac{x_j - x_i}{x_ix_j}\right)(x_i - x_j) = (x_i - x_j)(x_i^{-1} - x_j^{-1}).
    \end{equation*}
    This allows us to write
    \begin{align}
        \prod_{i \neq j} (1 - x_i/x_j) = \prod_{1\leq i < j \leq n} (x_i - x_j)\prod_{1\leq i < j \leq n} (x_i^{-1} - x_j^{-1}) =  a_{\delta^n}(x_1, \dots,x_n)\,a_{\delta^n}(x_1^{-1},\dots,x_n^{-1}),
    \end{align}
    where the second equality follows from \cref{lem:adelta}. Then,
    \begin{align*}
        \operatorname{CT}\left(s_{\lambda}(x)s_{\mu}(x^{-1})\prod_{i\neq j}(1 -  x_i/x_j)\right) &= \operatorname{CT}\left(s_{\lambda}(x)s_{\mu}(x^{-1})a_{\delta^n}(x)\,a_{\delta^n}(x^{-1})\right) \\ \eqalignref{eq:Laurent-Schur}\operatorname{CT}\left(a_{\lambda +\delta^n}(x)\,a_{\mu + \delta^n}(x^{-1})\right).
    \end{align*}
    From \cref{eq:aalpha}, we have
    \begin{align*}
        a_{\lambda +\delta^n}(x)\,a_{\mu + \delta^n}(x^{-1}) &= \left( \sum_{\sigma \in S_n} \sgn(\sigma)x^{(\lambda + \delta^n)^\sigma} \right)\left( \sum_{\tau \in S_n} \sgn(\tau)x^{-(\mu + \delta^n)^\tau} \right) \\
        &= \sum_{\sigma, \tau\,\in S_n} \sgn(\sigma)\sgn(\tau )x^{(\lambda + \delta^n)^\sigma - (\mu + \delta^n)^\tau}.
    \end{align*}
    For the constant term of the Laurent polynomial $a_{\lambda +\delta^n}(x)\,a_{\mu + \delta^n}(x^{-1})$, we look at all \(\sigma, \tau \in S_n\) such that \[(\lambda + \delta^n)^\sigma - (\mu + \delta^n)^\tau = 0.\]
    Then $(\lambda + \delta^n)^{\sigma \circ \tau^{-1}} =(\mu + \delta^n)$. As $\lambda+\delta^n,\mu+\delta^n \in \SPar_n$, each with no repeated elements, and $\sigma \circ \tau^{-1}$ is a set-wise bijection, $\lambda + \delta^n$ and $\mu + \delta^n$ must have the same parts. Since both are ordered, it follows that $\lambda + \delta^n=\mu + \delta^n$. Then by \cref{lem:lamdel}, $\lambda=\mu$, and so $\sigma \circ \tau^{-1}$ must be the identity, and $\sigma=\tau$. Thus, if \(\lambda \neq \mu\), we have
    \[
     \operatorname{CT}\left(s_{\lambda}(x)s_{\mu}(x^{-1})\prod_{i\neq j}(1 -  x_i/x_j)\right) = 0.
    \]
    If \(\lambda = \mu \), then by the above discussion, $\sigma = \tau$, and so we have
    \begin{align}
        \operatorname{CT}\left(\sum_{\sigma,\nu \in S_n} \sgn(\sigma)\sgn( \nu)x^{(\lambda + \delta^n)^\sigma - (\mu + \delta^n)^\nu} \right) = \sum_{\sigma \in S_n} \sgn(\sigma)^2 = |S_n| = n!,
    \end{align}
    as desired.
     \end{proof}
As the Laurent Schur polynomials form a basis for $\LSym_n$, the following corollaries hold.
\begin{cor}
    The pairing \cref{eq:Hall-IP} is positive-definite (i.e., $\langle f(x_1, \dots, x_n), f(x_1, \dots, x_n) \rangle_n > 0$, for \(f(x_1,\dots,x_n) \neq 0\)).
\end{cor}

\begin{rem}
     \cref{prop:SchurOGB} also implies the pairing is symmetric, which provides another proof for the fact that the pairing is symmetric, in addition to \cref{prop:bilin-symm}, where it is shown directly. Together with \cref{prop:bilin-symm}, we have shown the pairing \cref{eq:Hall-IP} is indeed an inner product, as promised.
\end{rem}

\begin{cor}
    The Hall inner product \cref{eq:Hall-IP} is a map \(\LSym_n \times \LSym_n \to \Z\). 
\end{cor}

\begin{cor}
    For all \(d,r \in \Z\) we have 
    \begin{align}
        \langle e_{n}^d, e_{n}^r\rangle_n \, = \delta_{dr}.
    \end{align} 
\end{cor}
\begin{proof}
    Since \(s_{0} = 1\) by \cref{prop:schuren}, we have \(e_n^k = s_{(k^n)}\).
\end{proof}

\begin{defin}
    Let
\begin{align}
    s_\lambda(x) = \sum_{\mu \in \SPar_n} K_{\lambda,\mu} m_\mu (x).
\end{align}
The $K_{\lambda,\mu}$ are called \emph{generalized Kostka numbers}.
This sum is finite, since Laurent monomial symmetric polynomials form a basis of $\LSym_n$.

\end{defin}
Recall $\lambda^\vee := (-\lambda_{n},\dots,-\lambda_{1})$ from \cref{lem:tech}.
Consider the ring homomorphism
\begin{align}
    \LPol_n \longrightarrow \LPol_n, \quad x_i\longmapsto x_i ^{-1}.\label{eq:gamma}
\end{align}

\begin{prop}\label{prop:gamma}
Let $\Gamma$ denote the restriction of \cref{eq:gamma} to $\LSym_n$. Then the following hold:
    \begin{enumerate}
        \item $\Gamma(s_{\lambda}(x)) = s_{\lambda^\vee}(x)$, \label{prop:gamma-a}
        \item $\Gamma(m_{\lambda}(x)) = m_{\lambda^\vee}(x)$,\label{prop:gamma-b}
        \item $\Gamma$ is an involution,\label{prop:gamma-c}
        \item $\Gamma$ is an isometry.\label{prop:gamma-d}
    \end{enumerate}
   Here, $x$ is a shorthand for $x_1, \dots, x_n$. 
\end{prop}
\begin{proof}
     Part \cref{prop:gamma-a} is a restatement of \cref{prop:relsch}. For \cref{prop:gamma-b}, recall by definition,
    \begin{align*}
    m_\lambda(x_1, \dots, x_n) =\sum_{\beta \in \lambda \mathfrak{S}_n} x^\beta, \quad \lambda \in \SPar_n(k).
\end{align*}
Then \[\Gamma(m_\lambda(x_1, \dots, x_n))=\sum_{\beta \in \lambda \mathfrak{S}_n} x^{-\beta}=\sum_{-\beta \in \lambda \mathfrak{S}_n} x^{\beta}=\sum_{\beta \in \lambda^\vee \mathfrak{S}_n} x^{\beta}.\] \cref{prop:gamma-c} is clear.  Finally, \cref{prop:gamma-d} follows from \cref{prop:gamma-a} since it shows that the image of an orthonormal basis is an orthonormal basis. 
\end{proof}

\begin{cor}
   Under $s_\lambda \mapsto s_{\lambda^\vee}$ and $m_\mu \mapsto m_{\mu^\vee}$, where $\lambda^\vee, \mu^\vee$ are as defined in \cref{lem:tech}, we have that \(K_{\lambda,{\mu}}=K_{\lambda^\vee,\mu^\vee}\).
\end{cor}

\section{Laurent symmetric functions} \label{sec:lsf}

\begin{defin}
    For an alphabet $X= (X_1, X_2, \dots) $, \emph{$\Lambda(X) $} denotes the \emph{ring of symmetric functions} in the variables $X_1, X_2 , \dots $.
\end{defin}

We use the term \textit{functions} to distinguish this ring from the ring of symmetric \textit{polynomials} in finitely many variables. 

\begin{defin}
    Let $X = (X_1, X_2, \dots) $ and $Y=(Y_1 , Y_2, \dots )$ be two independent alphabets. The ring of \emph{Laurent symmetric functions} in infinitely many variables is
    \begin{align}
        \LSym := \Lambda(X) \otimes_\Z \Lambda (Y)
    \end{align}
    
\end{defin}

\begin{rem}
    Since the tensor product commutes with direct sums, we have
    \begin{align}
        \Lambda(X) \otimes_\Z \Lambda (Y) = \bigoplus_{k,\ell \in \N} \Lambda^k (X) \otimes_\Z \Lambda^\ell (Y).
    \end{align}
The submodules \(\Lambda^k (X) \otimes_\Z \Lambda^\ell (Y)\) are of finite rank and yield a \(\N^2\) grading of \(\LSym\). We can use this to construct a $\Z$-grading
\[
\Lambda(X) \otimes_\Z \Lambda (Y) = \bigoplus_{r \in \Z} \bigoplus_{\substack{k - \ell = r\\ k,\ell \in \N}} \Lambda^k (X) \otimes_\Z \Lambda^\ell (Y),
\]
where the submodules of degree $r$ are \(\bigoplus_{\substack{k - \ell = r\\ k,\ell \in \N}} \Lambda^k (X) \otimes_\Z \Lambda^\ell (Y)\). Note that in this case the graded components are abelian groups of infinite rank.
\end{rem}

We extend the elementary symmetric functions $e_r(x_1, x_2, \dots)$ defined in \cite[Ch.~I, \S2, p.~19]{Mac15} to Laurent symmetric functions as follows.
\begin{defin} \label{def:er_xy}
    For $r \in \Z$ we define the \emph{Laurent elementary symmetric functions} in infinite variables as follows. 
    \begin{align*}
        e_r &= e_r(X), \qquad r> 0, \\
        e_{r} &= e_{-r}(Y), \qquad r< 0, \\
        e_0 &= 1,
    \end{align*}
    where $e_r(X), e_r(Y)$ for $r>0$ denote the elementary symmetric function in the variables $X_1, X_2, \dots $, $Y_1, Y_2, \dots$.
\end{defin}
By the fundamental theorem of symmetric functions
\cite[Ch.~I, \S2]{Mac15}, we have
$\Lambda(X)=\Z[e_1(X),e_2(X),\dots]$,
where $e_1(X),e_2(X),\dots$ are algebraically independent.
 It follows that
\begin{prop}\label{prop:ebasis}
    We have
    \begin{align*}
        \LSym = \Z[ e_1(X),e_2(X),\dots,e_1(Y),e_2(Y),\dots ],
    \end{align*}
where the elements \(e_1(X),e_2(X),\dots,e_1(Y),e_2(Y),\dots\) are algebraically independent.
With the notation of \cref{def:er_xy} we have
\begin{align}
    \LSym = \Z[ e_{\pm 1},e_{\pm 2},\dots]
\end{align}
\end{prop}

Similarly, we extend the definitions of the complete homogeneous symmetric functions $h_r$ and the power sums $p_r$ in \cite[Ch.~I, \S2 pp.21-23]{Mac15} to the Laurent symmetric functions. 
\begin{defin}\label{def:hr_pr}
    For $r \in \Z$, we define 
    \begin{align}
        h_r = \begin{cases}
            h_r(X), & r > 0, \\
            h_{-r}(Y), & r<0, \\
            1, &  r=0,
        \end{cases}
    \end{align}
    and 
    \begin{align}
        p_r = \begin{cases}
            p_r(X), &r > 0, \\
            p_{-r}(Y), & r<0.
        \end{cases}
    \end{align}
\end{defin}

Note we do not define \(p_0\).

\begin{defin}\label{def:rho}
    For $n \geq 1$, we define a ring homomorphism
    \begin{align*}
        \rho_n \colon \LSym \rightarrow \LSym_n
    \end{align*} 
    on the generators $e_r$ by
    \begin{align}
        \rho_n(e_r(X)) &= e_r(x_1, \dots , x_n), \label{eq:eX}
    \end{align}
    and
    \begin{align}
        \rho_n (e_r(Y)) = e_r(x_1^{-1}, \dots, x_n^{-1} ) \overset{(\star)}{=}  \begin{cases}
            e_{n-r}(x_1,\dots,x_n)e_n(x_1,\dots,x_n)^{-1},
                & 0\leq r\leq n,\\
            0. & r>n
        \end{cases} \label{eq:eY}
        \end{align}
        Since $e_r(X), e_r(Y)$ freely generate $\LSym$ as a $\Z$-algebra, \cref{eq:eX} and \cref{eq:eY} together define a unique ring homomorphism. 
\end{defin}

\begin{rem}
    The equality $(\star)$ follows from the definition that 
    \begin{align*}
        e_{n-r}(x_1, \dots, x_n)=\displaystyle\sum_{1 \leq i_1 < \cdots < i_{n-r} \leq n}x_{i_1}\cdots x_{i_{n-r}},
    \end{align*}
    and $e_n(x_1, \dots, x_n)=x_1\cdots x_n$, so we have
    \begin{align*}
        e_{n-r}(x_1,\dots,x_n)e_n(x_1,\dots,x_n)^{-1} 
        &= \left( \sum_{1 \leq i_1 < \cdots < i_{n-r} \leq n}x_{i_1}\cdots x_{i_{n-r}} \right)(x_1\cdots x_n)^{-1} \\
        &= \sum_{1 \leq i_1 < \cdots <  i_{r}\leq n} x_{i_1}^{-1}\cdots x_{i_r}^{-1} \\
        &= e_r(x_1^{-1}, \dots, x_n^{-1}).
    \end{align*}
\end{rem}

The alphabets $X$ and $Y$ are independent in the ring of Laurent symmetric \textit{functions}. Under the specialization to the finite variable case, however, they get mapped to $x_i$ and $x_i^{-1}$ respectively and thus become inverses to each other.

Since we are working with symmetric functions alongside symmetric polynomials, hereafter, we will denote elementary symmetric polynomials $e_r(x_1, \dots, x_n)$ by \emph{$e_r^{(n)}$}, to distinguish them from elementary symmetric functions, which we will denote by \emph{$e_r$}. 

\begin{defin}
For $r$ such that $0 \leq r \leq n$, let
\begin{align}
         e_{-r}^{(n)} : = e_r^{(n)} (x^{-1}). \label{eq:x-absent}
\end{align}
\end{defin}

It is desirable to have a description of \(\LSym\) as an inverse limit relating it to the finite variable case.

\begin{defin}
     For $d\geq 1$, set 
\begin{align*}
    {A_d }:= \Z[e_1(X) , \dots, e_d(X), e_1 (Y) , \dots, e_d(Y)] \subseteq \LSym.
\end{align*}
\end{defin}

\begin{prop}
    For $d \geq 1$ and $n \geq 2d +1$, 
    \begin{align*}
        \rho_n|_{A_d}\colon A_d\longrightarrow \LSym_n
    \end{align*}
    is injective.
\end{prop}

\begin{proof}
    Since \(e_1(X),\dots,e_d(X),e_1 (Y),\dots,e_d(Y)\) are algebraically independent in \(\LSym\), to show \(\rho_n|_{A_d}\) is injective, it suffices to check that, if \(n \geq 2d + 1\), the set
    \[\{\rho_n(e_1(X)),\dots,\rho_n(e_d (X)),\rho_n (e_1 (Y)),\dots,\rho_n  (e_d(Y))\}=\left\{e_1^{(n)},\dots, e_d^{(n)}, e_{n- 1}^{(n)} /e_n^{(n)},\dots,  e_{n - d}^{(n)}/e_n^{(n)} \right\}\] is algebraically independent in \(\LSym_n\).

    Let \(f(T_1, \dots, T_{2d}) \in \Z [T_1,\dots,T_{2d}]\) be such that \(f(e_1^{(n)},\dots, e_d^{(n)},e_{n - 1}^{(n)}/e_n^{(n)},\dots,e_{n -d}^{(n)}/e_n^{(n)}) = 0\). If we express \(f(T_1, \dots, T_{2d})\) as a finite sum
    \[
    f(T_1, \dots, T_{2d}) =\sum_{\alpha \in \N^{2d}} c_\alpha T_1^{\alpha_1} \cdots T_{2d}^{\alpha_{2d}}, \quad c_\alpha \in \Z,
    \]
    we have that
    \begin{align*}
    0 
    &= \sum_{\alpha \in \N^{2d}} c_\alpha \left(e_1^{(n)}\right)^{\alpha_1} \cdots \left(e_d^{(n)}\right) ^{\alpha_{d}} \left(\frac{e_{n -1}^{(n)}}{e_n^{(n)}}\right)^{\alpha_{d + 1}} \cdots \left(\frac{e_{n -d}^{(n)}}{e_n^{(n)}}\right)^{\alpha_{2d}}\\
    &=\sum_{\alpha \in \N^{2d}} c_\alpha \left(e_1^{(n)}\right) ^{\alpha_1} \cdots \left(e_d^{(n)}\right) ^{\alpha_{d}} \left(e_{n -1}^{(n)}\right) ^{\alpha_{d + 1}} \cdots \left(e_{n -d}^{(n)}\right) ^{\alpha_{2d}} \left(e_n^{(n)}\right) ^{-(\alpha_{d + 1} + \cdots \alpha_{2d})}.
    \end{align*}
    Because \(n  - d\geq d + 1\), the set
    \begin{align}
        \{\left(e_1^{(n)}\right) ^{\alpha_1} \cdots \left(e_d^{(n)}\right) ^{\alpha_{d}} \left(e_{n -1}^{(n)}\right) ^{\alpha_{d + 1}} \cdots \left(e_{n -d}^{(n)}\right) ^{\alpha_{2d}} \left(e_n^{(n)}\right) ^{-(\alpha_{d + 1} + \cdots \alpha_{2d})} \mid (\alpha_1,\dots,\alpha_{2d}) \in \N^{2d}\} \label{eq:elt-li}
    \end{align}
    is a subset of the basis $\cref{eq:elt-basis}$ in \cref{cor:FTLSP}. Thus, the set \cref{eq:elt-li} is linearly independent. As a consequence,  \(c_\alpha = 0\) for all \(\alpha \in \N^{2d}\). We conclude that the polynomial \(f(T_1, \dots, T_{2d})\) is the zero polynomial. 
\end{proof}

\begin{cor} \label{cor:giraffe}
For $n,d \in \N_{\geq 1}$, we denote by $A_{d,n}$ the image of $\rho_n|_{A_d}$ in $\LSym_n$. Then, \(A_{d,n} = \Z [e_{\pm1}^{(n)},\dots,e_{\pm d}^{(n)}] \subseteq \LSym_n\) and
    \begin{align*}
    \rho_n|_{A_d} \colon A_d \longrightarrow A_{d,n}
\end{align*}
is an isomorphism for \(n \geq 2d + 1\). 
\end{cor}

This motivates the following definition.

\begin{defin}
For \(m\geq n\geq 2d+1\), define
\begin{align}
     \pi_{m,n}^{(d)}\colon A_{d,m}\longrightarrow A_{d,n}, \quad  \pi_{m,n}^{(d)}(\rho_m(F))=\rho_n(F), \ F\in A_d.
\end{align}
That is, we define
    \begin{align}
        \pi_{m,n}^{(d)} := \rho_n|_{A_d} \circ (\rho_m|_{A_d})^{-1}. \label{eq:banana}
    \end{align}
\end{defin}
From \cref{eq:banana}, we see that \(\pi^{(d)}_{m,n}\) is a ring isomorphism for all \(m \geq n \geq 2d + 1\). By definition,
\[
\pi_{m,n}^{(d)} (e_i^{(m)}) = e_i^{(n)}, \qquad -d \leq i \leq d.
\]
In particular, \(\pi_{l,n}^{(d)} = \pi_{m,n}^{(d)} \circ \pi_{l,m}^{(d)} \) for \(l \geq m \geq n \geq 2d + 1\), and therefore the \(\{\pi_{m,n}^{(d)}\colon A_{d,m} \rightarrow A_{d,n}\}\) form an inverse system of rings. Because all \(\pi_{m,n}^{(d)}\) and \(\rho_n|_{A_d}\) are isomorphisms and they commute, i.e.
\[
\begin{tikzcd}
A_d \arrow[d, "\rho_m"] \arrow[rd, "\rho_n"] &           \\
{A_{d,m}} \arrow[r, "{\pi_{m,n}^{(d)}}"]     & {A_{d,n}},
\end{tikzcd}
\]
we get that 
\begin{align}
    A_d \cong \varprojlim_{n \geq 2d + 1} A_{d,n}.
\end{align}

Additionally, the \(A_d\) form a system of subrings of \(\LSym\) which exhaust the ring, that is, 
\begin{align}
    \bigcup_{d \geq 1} A_d = \LSym.
\end{align}
 This implies that \(\LSym = \varinjlim_{d \geq 1} A_d \). Thus, we have proven the following theorem.
\begin{theo}
    With the notation above, we have $ A_d \cong \varprojlim_{n\geq 2d+1} A_{d,n}$, and hence
    \begin{align}
        \LSym
        =
        \varinjlim_d A_d
        \cong
        \varinjlim_d \varprojlim_{n\geq 2d+1} A_{d,n},
    \end{align}
    where the inverse limit is with respect to the maps $\pi^{(d)}_{m,n}$ defined in \cref{eq:banana}. 
\end{theo}

We extend the \emph{generating functions} for the elementary symmetric functions and the homogeneous symmetric functions from \cite[Ch.~I, \S2 Eqs.~2.2, 2.5, 2.10]{Mac15} to Laurent symmetric functions. Define
\begin{align}
    E^+(t):=\sum_{r\geq0}e_r(X)t^r,\qquad
    E^-(t):=\sum_{r\geq0}e_r(Y)t^r,
\end{align}
and
\begin{align}
    H^+(t):=\sum_{r\geq0}h_r(X)t^r,\qquad
    H^-(t):=\sum_{r\geq0}h_r(Y)t^r.
\end{align}

\begin{prop}
We have
\begin{align}
    E^+(t)H^+(-t)=1,\qquad E^-(t)H^-(-t)=1.
\end{align}
\end{prop}
\begin{proof}
    Both identities follow from \cite[Ch.~I, \S2, Eq.~(2.6)]{Mac15}.
\end{proof}

As for the non-Laurent case, we define the Laurent symmetric functions over $\Q$ by
    \begin{align}
        \LSym_\Q := \Q \otimes_\Z \LSym.
    \end{align}

\begin{prop}
    There is an isomorphism
    \[
    \LSym_{\Q} \cong \Lambda_{\Q}(X) \otimes_\Q \Lambda_\Q (Y)
    \]
\end{prop}

\begin{proof}

    Using the integral isomorphism $\LSym \cong \Lambda(X) \otimes_\Z \Lambda(Y)$ together with the canonical base change isomorphism $R' \otimes_R (M \otimes_R N) \cong (R' \otimes_R M) \otimes_{R'} (R' \otimes_R N)$ for extension of scalars from $\Z$ to $\Q$, we obtain:
    \begin{align*}
    \LSym_\Q &= \Q \otimes_\Z \LSym \\
    &\cong \Q \otimes_\Z (\Lambda(X) \otimes_\Z \Lambda(Y)) \\
    &\cong (\Q \otimes_\Z \Lambda(X)) \otimes_\Q (\Q \otimes_\Z \Lambda(Y)) \\
    &= \Lambda_\Q(X) \otimes_\Q \Lambda_\Q(Y).
    \end{align*}
\end{proof}

Recall the definition of $p_r$ for $r \in \Z \setminus \{0\}$ in \cref{def:hr_pr}.

\begin{cor}
    Over \(\Q\), we have that
    \[
        \LSym_{\Q}
        =
        \Q[p_{\pm1}, p_{\pm 2}, \dots],
    \]
    where \(p_{\pm1}, p_{\pm 2}, \dots\) are algebraically independent.
\end{cor}

\begin{proof}
    Indeed, we know that \(\Lambda_\Q (X) = \Q[p_1,p_2,\dots]\) (see \cite[Ch.~I \(\S2\)]{Mac15}). Therefore, we have
    \[
    \LSym_\Q \cong \Q[p_1(X),p_2(X),\dots] \otimes_\Q \Q [p_1 (Y),p_2(Y),\dots] \cong \Q[p_1(X),p_2(X),\dots,p_1 (Y),p_2(Y),\dots].
    \]
\end{proof}

\begin{prop} \label{prop:prhon}
For $\rho_n$, with $n, r \geq 1$, we have 
    \[
        \rho_n(p_r)=\sum_{i=1}^n x_i^r,\qquad
        \rho_n(p_{-r})=\sum_{i=1}^n x_i^{-r}.
    \]
\end{prop}

\begin{proof}
    For $r\geq 1$, it suffices to show that the diagram
    \[
\begin{tikzcd}
\Lambda (X) \arrow[r, hook] \arrow[d, "\rho_n^X"'] & \Lambda (X) \otimes_{\mathbb{Z}} \Lambda (Y) \arrow[d, "\rho_n"] \\
\mathrm{Sym}_n \arrow[r, hook]     & \mathrm{LSym}_n,      
\end{tikzcd}
\]
commutes, where $\rho_n^X$ denotes the restriction to finitely many variables in the non-Laurent symmetric polynomials as given in \cite[Ch.~I, \S2, pp.18-19]{Mac15}. It suffices to show $\rho_n^X(e_r) = \rho_n(e_r)$ (since the $e_r$'s are generators and $\rho_n$ is a ring homomorphism). This follows immediately from the definition of $\rho_n$, since we have
\begin{align*}
    \rho_n^X(e_r) = e_r^{(n)}= \rho_n(e_r).
\end{align*}

We also note that the diagram
\begin{equation}
    \begin{tikzcd}
\Lambda(X) \otimes \Lambda (Y) \arrow[r, "\sigma"] \arrow[d, "\rho_n"'] & \Lambda(X) \otimes \Lambda(Y) \arrow[d, "\rho_n"] \\
\LSym_n \arrow[r, "\Gamma"']                   & \LSym_n 
\end{tikzcd} \label{eq:cd}
\end{equation}
commutes, where $\sigma$ is defined by $\sigma(X) = Y$ and $\sigma(Y) = X$ and $\Gamma$ is defined in \cref{eq:gamma}, giving symmetry in $X$ and $Y$.
The commutativity of the diagram can also be checked by considering the images of the generators \(e_r\) for \(r \in \Z\). Then, using \cref{eq:cd} and knowing that \(\rho_n(p_r) = p_r^{(n)}\) for \(r > 0\), we can deduce that images of \(p_r\) for \(r < 0\) under the \(\rho_n\) map are the ones stated.
\end{proof}

\section{Laurent Schur functions} \label{sec:lscf}

In the following we denote a signed partition by a pair of partitions $\lambda, \mu \in \Par$. Recall that for any \(n \geq \ell(\lambda) + \ell(\mu)\) the signed partition can be expressed as an \(n\)-tuple
\begin{align*}
    (\lambda_1 ,\dots , \lambda_r, 0 , \dots, 0 , -\mu_s , \dots, -\mu _1),
\end{align*}
where $r = \ell(\lambda) $ and $s = \ell (\mu)$. Accordingly, for $n \geq r+s$, let  \begin{align*}
        s_{\lambda,\mu}^{(n)} : = s_{(\lambda_1 ,\dots , \lambda_r, 0 , \dots, 0 , -\mu_s , \dots, -\mu _1)} (x_1, \dots,x_n),
    \end{align*}
where the right hand side is a Laurent Schur polynomial in $n$ variables as defined in \cref{def:schur}. The number of zeros added in the middle of the partition is $n - \ell(\lambda) - \ell(\mu)$.

For $\lambda, \mu \in \Par$, let $s_{\lambda/\mu}$ denote the \emph{skew Schur function} by
    \begin{align}
    s_{\lambda/\mu}:=\sum_{\nu \in \Par}\langle s_{\lambda}, s_\mu s_\nu \rangle s_\nu,
    \end{align}
which lies in the ring of symmetric functions (see \cite[Ch.~I, \S5, Eq.~(5.3)]{Mac15}). Here, $\langle \cdot , \cdot \rangle$ is the Hall inner product for usual symmetric functions (i.e.\ defined such that the Schur functions are orthonormal). Likewise, let $\supseteq$ be the relation given by
\begin{align}
    \lambda \supseteq \mu \iff \lambda_i \geq \mu_i \text{ for all } i. \label{eq:cont}
\end{align}
Equivalently,  $\lambda \supseteq \mu$ if and only if the Young diagram of $\mu$ embeds in that of $\lambda$. For example, $\lambda=(4, 3, 2, 1)$ and $\mu=(3, 2, 1)$ are embedded as follows.
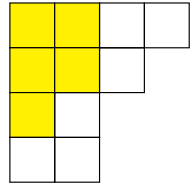
\[\begin{ytableau}
    *(Yellow)~ & *(Yellow) ~ &  ~ & ~ \\
    *(Yellow)~ & *(Yellow) ~ & ~\\
    *(Yellow)~ & ~ \\
    ~ & ~\\
\end{ytableau}\]
 Thus, we may associate a \emph{skew Young diagram} with $\lambda/\mu$, as the following example shows.
\[\begin{ytableau}
    \none~ & \none & ~ & ~ \\
    \none & \none & ~\\
    \none  & ~ \\
    ~ & ~\\
\end{ytableau}\]

\begin{defin}\label{def:lsf}
    For $\lambda,\mu \in \Par$ we define the \emph{Laurent Schur function $s_{\lambda,\mu}$} by 
    \begin{align}
        s_{\lambda,\mu}
        :=
        \sum_{\nu \in \Par}(-1)^{|\nu|}
        s_{\lambda/\nu}(X)\,
        s_{\mu/\nu'}(Y). \label{eq:ksf}
    \end{align}
\end{defin}
The sum is finite as there are only finitely many $\nu $ with $\nu \subseteq \lambda$ and $\nu' \subseteq \mu$ \cite[Ch.~I, \S5, Eq.~(5.7)]{Mac15}. 

\begin{prop}[{\cite[Prop. 2.2, Thm.~2.3]{Koi89}}]\label{prop:rhoschur}
    The Laurent Schur functions are well-behaved with respect to the restriction to finitely many variables. That is, for $n \geq \ell(\lambda) + \ell(\mu)$, we have
    \begin{align}
        \rho_n(s_{\lambda,\mu}) = s_{\lambda,\mu}^{(n)}.
    \end{align}
\end{prop}

\begin{prop}[{\cite[Cor.~2.3.2]{Koi89}}]\label{prop:sb}
The Laurent Schur functions
        $ s_{\lambda,\mu}$, $\lambda,\mu\in\Par$,
    form a basis of $\LSym$. 
\end{prop}

These two propositions motivate the following definition of the Hall inner product in the ring of Laurent symmetric functions.
\begin{defin}
    We define the \emph{Hall inner product} on $\LSym$ such that
    \begin{align}
        \langle s_{\lambda,\mu},s_{\alpha,\beta}\rangle = \delta_{\lambda\alpha}\delta_{\mu\beta}. \label{eq:Hallip-lsc}
    \end{align}
    As the Laurent Schur functions form a basis of $\LSym$, the pairing \cref{eq:Hallip-lsc} uniquely defines an inner product. 
\end{defin}

\begin{prop}
    Let 
\begin{align}
    F=\sum \alpha_{ (\lambda, \mu)} s_{\lambda, \mu}, \quad G=\sum \gamma_{ (\beta, \eta)} s_{\beta, \eta} \in \LSym. \label{eq:FGschurexp}
\end{align} 
Then, for all
$n \geq \mathrm{max}(\ell(\lambda)+\ell(\mu), \ell(\beta)+\ell(\eta))$ where the maximum is taken over all pairs $\lambda, \mu$ and $\beta, \eta$ for which $\alpha_{(\lambda, \mu)} \neq 0$ and $\gamma_{ (\beta, \eta)}\neq 0$,
we have
    \begin{align}
        \langle F,G\rangle =  \langle \rho_n(F), \rho_n(G) \rangle_n. \label{eq:isom}
    \end{align} 
\end{prop}
\begin{proof}
Let $n$ be as stated above.
We compute
    \begin{align*}
    \langle F,G\rangle &=
    \left \langle \sum_{\lambda,\mu \in \Par} \alpha_{(\lambda,\mu)}s_{\lambda,\mu},\sum_{\beta,\eta \in \Par} \gamma_{(\beta,\eta)}s_{\beta,\eta}\right \rangle \\
    &= \sum_{\lambda,\mu \in \Par}\sum_{\beta,\eta \in \Par} \alpha_{(\lambda,\mu)}\gamma_{(\beta,\eta)}\langle s_{\lambda,\mu}, s_{\beta,\eta}\rangle \\
    \eqalignref{eq:Hallip-lsc}\sum_{\lambda,\mu \in \Par}\alpha_{(\lambda,\mu)}\gamma_{(\lambda,\mu)}.
    \end{align*}
    On the other hand, by \cref{def:rho}, $\rho_n$ is linear, so we have
    \[
    \rho_n(F) = \rho_n\left(\sum_{\lambda,\mu \in \Par} \alpha_{(\lambda, \mu)} s_{\lambda,\mu}\right) = \sum_{\lambda,\mu \in \Par} \alpha_{(\lambda, \mu)}\rho_n(s_{\lambda,\mu}),
    \]
    similarly 
    \[\rho_n(G)=\sum \gamma_{ (\beta, \eta)} s_{\beta, \eta}^{(n)}.\]
    
    Hence,
    exploiting the orthogonality of the Schur polynomials in $n$ variables with respect to $\langle \cdot , \cdot \rangle_n$,
    the right hand side of \cref{eq:isom} becomes
    \begin{align*}
        \langle \rho_n(F),\rho_n(G)\rangle_n &=
    \left\langle \rho_n\left(\sum_{\lambda,\mu \in \Par} \alpha_{(\lambda,\mu)}s_{\lambda,\mu}\right),
    \rho_n\left(\sum_{\beta,\eta \in \Par} \gamma_{(\beta,\eta)}s_{\beta,\eta}\right)\right \rangle_n \\
    &= \sum_{\lambda,\mu \in \Par}\sum_{\beta,\eta \in \Par} \alpha_{(\lambda,\mu)}\gamma_{(\beta,\eta)}\langle \rho_n(s_{\lambda,\mu}), \rho_n(s_{\beta,\eta})\rangle_n \\ &\overset{(*)}{=}\sum_{\lambda,\mu \in \Par}\sum_{\beta,\eta \in \Par} \alpha_{(\lambda,\mu)}\gamma_{(\beta,\eta)}\langle s_{\lambda,\mu}^{(n)}, s_{\beta,\eta}^{(n)}\rangle_n \\
    &= \sum_{\lambda,\mu \in \Par}\alpha_{(\lambda,\mu)}\gamma_{(\lambda,\mu)},
    \end{align*}
    where we invoke in $(*)$ \cref{prop:rhoschur} and \cref{eq:SchurOGB}.
\end{proof}

\section{Representations of the general linear group}\label{sec:repGLn}

Throughout this section all vector spaces will be over \(\C\). Also we write \(\GL_n = \GL(n,\C)\).

\begin{defin}\label{def:rep}
    A \emph{representation} (over \(\C\)) of a group $G$ is a pair $(\rho, V)$, where $V$ is a vector space and $\rho$ is a group homomorphism 
    \begin{equation}
        \rho\colon G \longrightarrow \mathrm{GL}(V), \quad g \longmapsto \rho(g).  \label{eq:rho}
    \end{equation}
    When $\rho$ and $V$ are understood, we denote $\rho(g) \cdot v$ by $g \cdot v$, and refer to $\rho$ (or $V$) as the representation. 
\end{defin}

Note that if $V$ is finite dimensional, then by choosing a basis of \(V\), \cref{eq:rho} is the same as a group homomorphism \(\rho \colon G \rightarrow \GL_m \), where \(m = \dim V\). 

\begin{rem}
    The category of representations of a group $G$ can be identified with the category of $G$-modules. This identification can be used to define subrepresentations, quotients, tensor products and sums. However, we will give the definitions explicitly here, for greater clarity and note that they behave well with respect to this identification. 
\end{rem}

\begin{eg} \leavevmode
\begin{enum}
    \item For any $G$ and $V$, the homomorphism $\rho: g \mapsto \mathrm{Id}$ for all $g \in G$, where $\mathrm{Id}$ is the identity linear transformation $V \to V$, is the \emph{trivial representation}.
    \item Let $V = \C^n$ and $\rho = \mathrm{Id} \colon \GL_n \rightarrow \GL_n $. The pair $(\C^n, \mathrm{Id})$ is said to be a \emph{standard representation} of $\mathrm{GL}_n$.
\end{enum}
\end{eg}

\begin{defin}
    If $W \subseteq V$ is a vector subspace of $V$ invariant under $\rho_W(g) := \rho(g) \big|_W$ for all $g \in G$, then $(\rho_W, W)$ is a \emph{subrepresentation} of $(\rho, V)$. A representation $V$ is said to be \emph{irreducible} if it has no nontrivial subrepresentations. 
\end{defin}
\begin{defin}
    If $W$ is a subrepresentation of $V$, then $\rho$ induces an action on the quotient space $V/W$ given by 
    \begin{align}
        \tilde{\rho}(g) (v+W):=\rho(g)v+W. \label{eq:quo-rho}
    \end{align}
    The quotient space $V/W$ together with $\tilde\rho $ is said to be the \emph{quotient representation} $V/W$. 
\end{defin}
\begin{defin}\label{def:isom-rep}
    Two representations $(\rho, V)$ and $(\sigma, W)$ of $G$ are said to be \emph{isomorphic} if there exists a linear isomorphism $f\colon V \to W$ such that $f \circ \rho(g) =\sigma(g)\circ f$ for all $g \in G$. 
\end{defin}

\begin{defin}
    If $H \subseteq G$ is a subgroup of $G$, and $\rho: G \to \GL(V)$ is a representation of $G$, then $\rho$ restricts to a representation of $H$ as 
    \begin{equation}
        \rho|_{H}: H \to \GL(V), \quad h \mapsto \rho(h),
    \end{equation}
    which is said to be the \emph{restriction representation}, denoted $\mathrm{Res}^G_H \rho=\rho|_H$.
\end{defin}

\begin{defin} \label{def:pol/rat-rep}
    In finite dimensions, a representation $\rho\colon \GL_n \to \mathrm{GL}(V)$ of $\GL_n$ is said to be \emph{polynomial (resp.\ rational)} if, given a basis of $V$ with $m$ elements, the entries of the matrix of $\rho(g)$ are polynomial (resp.\ rational) in the matrix entries \(g_{i,j}\) for all $g \in \GL_n$. That is, for $1\leq i, j \leq m$, we have
    \begin{equation}
        \rho(g)_{ij}=f(g_{11}, g_{12}, \dots, g_{nn})
    \end{equation}
    for some $f \in \C[X_{11}, \dots, X_{nn}]$ (resp.\ $\C[X_{ij}, \det(X)^{-1}]$). Note that we use $\C[X_{ij}, \det(X)^{-1}]$ instead of $\C(X_{11}, \dots, X_{nn})$ because a rational representation must have matrix coefficients that are regular.
\end{defin}

\begin{rem}
    \cref{def:pol/rat-rep} is well defined and independent of the choice of basis since changing the basis yields a matrix whose coefficients are linear combinations of the coefficients of the matrix in the previous basis, and linear combinations of polynomial (resp.\ rational) functions are polynomial (resp.\ rational) functions. 
\end{rem}

\begin{defin}\label{def:char}
    Given a representation \(\rho\colon \GL_n \rightarrow \GL(V)\), with $V$ a finite-dimensional vector space, we define its \emph{character} \(\chi_\rho\) by
    \[
    \chi_\rho (g) = \mathrm{Tr}(\rho (g)),
    \]
    where \(\mathrm{Tr}\) denotes the trace of the matrix $\rho(g)$. In the following, given \(x_1,\dots,x_n \in \C^\ast\) then \(\chi_\rho (x_1,\dots,x_n)\) will be a shorthand for \(\chi_\rho (\mathrm{diag}(x_1,\dots,x_n))\), where \(\mathrm{diag}(x_1,\dots,x_n)\) is the diagonal matrix with entries \(x_1,\dots,x_n\).
\end{defin}

\begin{rem}
    The character is well defined because it is independent of the choice of the basis of $V$. If $C \rho(g) C^{-1}$ is the matrix representation of $\rho(g)$ in another basis, \[\mathrm{Tr}(C \rho(g) C^{-1})=\mathrm{Tr}(C^{-1} C \rho(g))=\mathrm{Tr}(\rho(g))\] by the cyclic property of the trace. 
\end{rem}

\begin{defin}\label{def:weight}
    Let $H\subseteq \GL_n$ be the subgroup of diagonal matrices, and $\rho \colon \GL_n \rightarrow \GL(V)$ be a representation of $\GL_n$. Let $t =$ diag$(t_1,\dots,t_n) \in H$.
    For any tuple $\alpha \in \Z^n$, we define an $\alpha$-induced function
    \begin{align*}
        \tilde\alpha \colon H &\longrightarrow \C,\\
        t &\longmapsto \tilde\alpha(t) = t_1^{\alpha_1}\cdots t_n^{\alpha_n}.
    \end{align*}
We say $v \in V, v \neq 0$ is a \emph{weight vector} with \emph{weight} $\alpha$ if
\begin{align}
    t\cdot v = \tilde\alpha (t) v = t_1^{\alpha_1}\cdots t_n^{\alpha_n} v \text{ for all } t \in H.
\end{align}
Let $V_\alpha := \{v \in V\,  \mid t \cdot v \, = \, \tilde\alpha(t) v, \text{ for all } t\in H \}$. If $V_\alpha \neq \{0\}$,  we say $V_\alpha$ is the \emph{weight space} of weight $\alpha$. 

Let $B$ be the subgroup of all upper triangular matrices in $\GL_n$, called the \emph{Borel subgroup}. A weight vector $v$ in a representation $V$ is called a \emph{highest weight vector} if
\begin{align*}
    B  v \subseteq \C^* v.
\end{align*}
\end{defin}
\begin{rem}\label{rk:starfruit}
    Every finite dimensional rational representation $V$ of $\GL_n$ can be described as a direct sum of its weight spaces,
\begin{align}
    V = \bigoplus_{\alpha \in \Z^n} V_\alpha, \label{eq:decomp}
\end{align}
(see \cite[\S 8.2, p.~116]{Ful97}).
Using \cref{eq:decomp}, the character $\chi_V (x)$ on $H$ may be expressed as
\begin{align}
    \chi_{V}(x) = \sum_{\alpha \in \Z^n}\dim(V_\alpha) x_1^{\alpha_1}\cdots x_n^{\alpha_n}, \quad x \in H. \label{eq:char}
\end{align}
For a proof, see, for example, \cite[\S 8.2, p.120]{Ful97}. Thus, we can regard the characters as Laurent polynomials. 
\end{rem}

\begin{prop}[{\cite[\S 8.2, Th.~2]{Ful97}}]
The rational irreducible representations of $\GL_n$ are indexed by their highest weights $  \alpha = (\alpha_1 \geq \dots \geq \alpha_n)$, where $\alpha_i \in \Z$. 
\end{prop}

\begin{eg}\label{ex:siamese}
Consider $\mathrm{GL}_2(\C)$, and $\C[x,y]_m$ the vector space of homogeneous polynomials of degree \(m\), which is generated by \(\{x^m,x^{m-1}y,\dots,xy^{m-1},y^m\}\). The elements of the vector space are given by
\begin{align*}
    p(x,y) = \sum_{i=0}^ma_ix^iy^{m - i} \quad \text{with } a_i \in \C.
\end{align*}
By definition, elements of $\mathrm{GL}_2$ are of the form
\begin{align*}
    g = \begin{pmatrix}
        a & b\\c & d
    \end{pmatrix}, \quad\text{ with } ad - bc \neq 0.
\end{align*}
 We can define an action of $\mathrm{GL}_2$ on $\C[x,y]_m$ by 
 \begin{align}
     \begin{pmatrix}
         a&b\\c&d
     \end{pmatrix} \cdot f(x, y)
     = f(ax + cy, bx + dy). \label{eq:action}
 \end{align}
Let $T \subset \mathrm{GL}_2(\C)$ be the torus of diagonal matrices ($b, c = 0$). For any $g = \begin{pmatrix} a & 0 \\ 0 & d \end{pmatrix} \in T$, the action yields $x \mapsto ax, y \mapsto dy$. 
Thus, for every $i \in \{0,\dots,m\}$, the elements $k\,x^iy^{m - i}$ ($k \in \C^*$) are weight vectors of weight $(i, m-i)$. 

Indeed, these are the only weight vectors for $T$. If a general polynomial $p(x,y) = \sum_{j=0}^m c_j x^j y^{m-j}$ is a weight vector for $T$, there must exist a character $\lambda \colon T \to \C^*$ such that $g \cdot p(x,y) = \lambda(g) p(x,y)$ for all $g \in T$. This requires
\begin{align*}
    \sum_{j=0}^m c_j a^j d^{m-j} x^j y^{m-j} = \sum_{j=0}^m c_j \lambda(g) x^j y^{m-j} \quad \forall a, d \in \C^*,
\end{align*}
which implies $c_j a^j d^{m-j} = c_j \lambda(g)$ for all $j$. Therefore, all non-zero terms in $p(x,y)$ must correspond to the exact same weight character $a^j d^{m-j}$ for every $a, d \in \C^*$. This holds simultaneously for all elements of $T$ if and only if $c_j = 0$ for all but a single index $i$. Thus, non-zero monomials $k\,x^iy^{m - i}$ are the only weight vectors of $T$.

Note that in particular, elements of the form $k x^m$ are highest weight vectors, since $a \neq 0$ for all $g \in B$. We claim there is no other highest weight vector, for suppose $k\,x^iy^{m - i}$ is a weight vector of highest weight, then for all
\begin{equation*}
    g=\begin{pmatrix}
         a & b \\
         0 & d
     \end{pmatrix} \in B,
\end{equation*}
we have 
\begin{equation}
    \lambda_g \, k\,x^iy^{m - i} =  g \cdot k\,x^iy^{m - i} = k\,a^ix^i(bx+dy)^{m-i}, \quad \lambda_g \in \C^*. \label{eq:peach}
\end{equation}
If $i<m$, $(bx+dy)^{m-i}$ expands into $\sum_j (bx)^j (dy)^{m-i-j}$, but $\lambda_g \, k\,x^iy^{m - i}$ is a monomial. Since $a, b, d$ are arbitrary, the first equality in \cref{eq:peach} does not hold unless $i=m$.
\end{eg}
In \cref{ex:siamese}, the representation with the map in \cref{eq:action} and the vector space $\C[x, y]_m$ is what we will later call \(\Sc^{(m)} \C^2\).

For \(\lambda\in\Par_n\), let \([\lambda]\) denote the set of boxes in the
Young diagram of \(\lambda\). For a vector space \(V\), define
\[
    V^{\times\lambda}:=\prod_{b\in[\lambda]} V .
\]
We refer to elements of \(V^{\times\lambda}\) as \emph{fillings} of the Young diagram
of \(\lambda\) by vectors of \(V\), i.e. an assignment of a vector \(v_b\in V\)
to each box \(b\in[\lambda]\). Choose a numbering $T$ of the boxes of the Young diagram of $\lambda$ by the numbers $1,\dotsc,r$, where \(r = |\lambda|\). Then from a filling $(v_1, \dots, v_r)$, where $v_i, \, 1\leq i \leq r$ is the vector in the box labeled by $i \in \{1, \dots, r\}$, we have a map 
\begin{align}
    V^{\times \lambda} \longrightarrow V^{\otimes r}, \quad (v_1, \dots, v_r) \longmapsto v_1 \otimes \cdots \otimes v_r.
\end{align}

\begin{defin}\label{def:ysds}
Fix a numbering of the $r$ boxes of $T$ by the set $\{1,\ldots,r\}$.
For each row $R$ of $T$, let $I_R\subseteq\{1,\ldots,r\}$ denote the
set of labels of the boxes in $R$, and for each column $C$ of $T$, let
$I_C\subseteq\{1,\ldots,r\}$ denote the set of labels of the boxes in $C$. Define  $R_T$ and $C_T$ to be the subgroups of $\mathfrak S_r$ preserving
the rows and columns of $T$ setwise, respectively. That is,
\begin{subequations}
\begin{align}
    R_T &= \{\sigma\in\mathfrak S_r : \sigma(I_R)=I_R
    \text{ for every row }R\text{ of }T\}, \\
    C_T &= \{\sigma\in\mathfrak S_r : \sigma(I_C)=I_C
    \text{ for every column }C\text{ of }T\}.
\end{align}
\end{subequations}
Let the symmetric group $\mathfrak S_r$ act on $V^{\otimes r}$ by
permuting tensor factors. Define endomorphisms of $V^{\otimes r}$
    \begin{align}
        a_T = \sum_{\sigma \in R_T} \sigma,
        \qquad
        b_T = \sum_{\sigma \in C_T} \sgn(\sigma)\sigma.
    \end{align}
Thus $a_T$ symmetrizes tensor factors lying in the same row of $T$, while $b_T$ antisymmetrizes tensor factors lying in the same column of $T$. 

For each pair $a_T, b_T$, the corresponding \emph{Young symmetrizer} is
    \begin{align}
         y_T = b_T a_T \in \C \mathfrak{S}_r.
    \end{align}
\end{defin}
Note that some authors, such as Weyman in \cite{Wey03}, define the Young symmetrizer as \[y_T = a_T b_T \in \C \mathfrak{S}_r,\] which gives a Schur module isomorphic to ours. 
\begin{defin}
    We define the \emph{Schur module} to be
    \begin{align}
        \Sc^\lambda V
        =
        \im\left(y_T \colon V^{\otimes r} \longrightarrow V^{\otimes r}\right). \label{eq:sm}
    \end{align}
\end{defin}

Later we will define a \(\GL (V)\)-representation on the vector space $\Sc^\lambda(V)$, hence the term \textit{Schur module}.

\begin{rem}
     Since $\im(y_T)$ is naturally isomorphic to $V^{\otimes r}/\ker(y_T)$, \cref{eq:sm} realizes $\Sc^\lambda V$ as a subquotient of $V^{\otimes r}$.
\end{rem}

\begin{prop}[{\cite[\S 4.1, Eq. ~(4.12), \S 4.2, Lem.~4.26]{Ful91}}]
    The element $y_T$ is a nonzero scalar multiple of an idempotent,
    \begin{align}
        y_T^2 = h_\lambda y_T,
    \end{align}
    where $h_\lambda$ is the product of the hook lengths (see \cref{def:hl}) of the boxes of $\lambda$. 
\end{prop}

It follows that one may equivalently use the idempotent
\begin{align}
    e_T = \frac{1}{h_\lambda} y_T,
\end{align}
to get
\begin{align}\label{eq:eTV}
    \Sc^\lambda V
        =
        e_T\left(V^{\otimes r}\right).
\end{align}

\begin{prop}
Given a linear map \(\phi \colon V \rightarrow W\), define 
\begin{align}
    \phi^{\otimes r}\colon V^{\otimes r} \longrightarrow W^{\otimes r}, \quad v_1 \otimes \cdots \otimes v_r \longmapsto \phi(v_1) \otimes \cdots \otimes \phi(v_r),
\end{align}
and extend linearly. Let
\begin{align}
\Sc^\lambda(\phi)\colon \Sc^\lambda(V) \longrightarrow  \Sc^\lambda(W
), \quad \Sc^\lambda(\phi)=\phi^{\otimes r}|_{\Sc^\lambda(V)}. \label{eq:sf}
\end{align}
The map \cref{eq:sf} is well defined and the assignments 
$V \mapsto \Sc^\lambda(V), \phi \mapsto \Sc^\lambda(\phi)$
define a functor.
\end{prop}
\begin{proof}
Let $\xi \in \Sc^\lambda(V)$. We want to show that \cref{eq:sf} is well defined. That is,  
\begin{align*}
    \left(\phi^{\otimes r}\right)(\xi) \in \Sc ^\lambda (W)=y_{T, W}(W^{\otimes r}).
\end{align*}
Since $\xi \in y_{T, V}(V^{\otimes r})$, there exists some $\zeta \in V^{\otimes r}$ such that $\xi=y_{T, V}(\zeta)$. If $\phi^{\otimes r}\circ y_{T, V}= y_{T, W} \circ \phi^{\otimes r}$, then
\begin{align*}
\left(\phi^{\otimes r}\right) (\xi) &=\left(\phi^{\otimes r}\circ y_{T, V}\right)(\zeta) \\
&= \left(y_{T, W} \circ \phi^{\otimes r}\right)(\zeta) \\
&= y_{T, W} \left( \phi^{\otimes r}(\zeta)\right) \in \Sc^\lambda(W),
\end{align*}
where we have used that $\phi^{\otimes r}(\zeta) \in W^{\otimes r}$. 
It remains to show $\phi^{\otimes r}\circ y_{T, V}= y_{T, W} \circ \phi^{\otimes r}$. Indeed, since the Young symmetrizers are linear combinations of $\sigma \in \mathfrak{S}_r$, we only need to show $\phi^{\otimes r}$ commutes with $\sigma$. As both $\phi^{\otimes r}$ and $\sigma$ are linear, it suffices to show that the commutation holds on an arbitrary pure tensor $v_1 \otimes \cdots \otimes v_r \in V^{\otimes r}$. But because $\phi^{\otimes r}$ acts by applying $\phi$ to each factor $v_i, 1 \leq i \leq r$ in the tensor power $v_1 \otimes \cdots \otimes v_r$, clearly $\phi^{\otimes r}$ commutes with permutation of the indices. Thus $\phi^{\otimes r}\circ y_{T, V}= y_{T, W} \circ \phi^{\otimes r}$, as desired. 

Finally, it is easily checked that \(\Sc^\lambda (\mathrm{Id}_V) = \mathrm{Id}_{\Sc^\lambda V}\) and \(\Sc^\lambda (\varphi \circ \psi ) = \Sc^\lambda(\varphi) \circ \Sc^\lambda (\psi)\) for any linear maps \(\psi \colon V \rightarrow W\) and \(\varphi \colon W \rightarrow U\) between vector spaces \(V,W\), and \(U\).
\end{proof}

\begin{defin}
    For a representation $(\rho, V)$, a representation on the Schur module $\Sc^\lambda(V)$ is given by 
    \begin{equation}
        \mathbb{S}^\lambda(\rho) \colon \GL_n \rightarrow \GL (\mathbb{S^\lambda }V), \quad \mathbb{S}^\lambda(\rho)(g) := \mathbb{S}^\lambda(\rho(g)),
    \end{equation} where the right hand side is defined by \cref{eq:sf}.
\end{defin}

\begin{rem}
    There are, in fact, many other equivalent ways to define the Schur module (see, for example, \cite[\S 8]{Ful97}). It is shown (up to the Young symmetrizer convention) in \cite[Lem.~2.2.13]{Wey03} that \cref{def:ysds} is equivalent to the one given in \cite[\S 8]{Ful97}.
\end{rem}

\begin{prop}
    The $\GL_n$-module \cref{eq:eTV} is independent of the choice of $T$ up to isomorphism.
    \end{prop}
    \begin{proof}This follows from the fact that \cref{eq:eTV} is the unique (up to isomorphism) irreducible polynomial $\GL_n$-module of highest weight $\lambda$, provided $\ell(\lambda) \leq n$ {\cite[\S 8.2 Thm. 2]{Ful97}}.  If $\ell(\lambda) > n$, then $\Sc^\lambda V = 0$. 
\end{proof}

\begin{theo}[{\cite[\S8.3 (14)]{Ful97}}]\label{prop:char-sc}
    Let $\lambda \in \Par_n$. Then for all \(x_1,\dots,x_n \in \C^\ast\) we have 
    \begin{align}
        \chi_{\mathbb{S}^\lambda V_n}(\mathrm{diag}(x_1,\dots,x_n)) = s_\lambda(x_1,\dots,x_n) \label{eq:char-sc}
    \end{align}
\end{theo}

\begin{defin}
    For \(k \in \Z\), the \emph{determinant representations} over \(\C\), denoted by \(\det^k\), are given by
\begin{align}
    g \cdot v = \det(g)^k v \label{eq:det-rep}, \qquad \text{for all } v \in \C.
 \end{align}
\end{defin}
 Note that for \(k \in \N\), this is a polynomial representation, and for \(k \in \Z_{<0}\) 
it is a rational representation, as in \cref{def:pol/rat-rep}.

    \begin{defin}
    Let $V_1, V_2$ be two vector spaces equipped with representations $\rho_1 \colon \GL_n \rightarrow \GL(V_1) $ and $\rho_2 \colon \ GL_n \rightarrow \GL(V_2) $. The \emph{tensor product} of these representations is the representation $\rho \colon \GL_n \rightarrow \GL(V_1 \otimes V_2)$ given by 
    \begin{align}
        \rho(g) (v_1 \otimes v_2) = \rho_1(g)(v_1) \otimes \rho_2(g)(v_2).
    \end{align}
\end{defin}

\begin{prop}\label{prop:isom}
    Let $\lambda, \widehat{\lambda} \in \Par_n$ and $k,\widehat{k} \in \Z$ such that $\lambda + (k^n) = \widehat{\lambda} + (\widehat{k}^n) $. Then 
    \begin{align}
        \Sc^\lambda V_n\otimes \det\nolimits^k \cong
        \Sc^{\widehat{\lambda}} V_n \otimes \det\nolimits^{\widehat{k}}.
    \end{align}
\end{prop}
\begin{proof}
    This is an immediate consequence of \cite[\S 8.2, Th.~2]{Ful97} (note that Fulton denotes $\Sc^\lambda V_n$ by $E^\lambda$ and $\mathrm{det}^k$ by $D^{\otimes k}$).   
\end{proof}

    \begin{defin}\label{def:elratrep}
    For \(\lambda \in \SPar_n\) we define the \emph{irreducible rational representation $R_n^\lambda$ of highest weight $\lambda$} as 
    \begin{equation}
        \mathrm{R}_n^\lambda = \Sc^{\widehat{\lambda}} V_n \otimes \det^k \label{eq:turquoise}
    \end{equation}
    where \(\widehat{\lambda} \in \Par_n\) and \(k \in \Z\) are such that \(\lambda = \widehat{\lambda} + (k^n)\). 
\end{defin}

Note that $k \in \Z$ can be negative, so $\lambda=\widehat{\lambda}+(k^n)$ can indeed be every signed partition in $\SPar_n$. Moreover, $\mathrm{R}^\lambda$ is well defined up to isomorphism by \cref{prop:isom}. 

\begin{prop}[{\cite[\S 8.2]{Ful97}}]
    For every \(\lambda \in \SPar_n\), the representation \(\mathrm{R}_n^\lambda\) of \(\GL_n\) is irreducible. Furthermore, every irreducible rational representation of \(\GL_n\) is isomorphic to one of the \(\mathrm{R}_n^\lambda\).
\end{prop}
Thus, the irreducible rational representations of $\GL_n$ are indexed by signed partitions of length at most $ n$.

\begin{prop}
    For every $\lambda \in \SPar_n$ and \(x_1,\dots,x_n \in \C^\ast\) the character of the irreducible rational representation \(\mathrm{R}^\lambda\) is given by
    \begin{align}
        \chi_{R_n^\lambda} = s_\lambda
    \end{align} 
    as Laurent polynomials in $x_1, \dots, x_n$.
\end{prop}
\begin{proof}
    Recall \(\mathrm{R}^\lambda = \Sc^{\widehat{\lambda}} V_n \otimes \det^k \). For any representations \(\rho_V \colon \GL_n \longrightarrow \GL(V)\), $ \rho_W \colon \GL_n \rightarrow \GL(W)\) we have $\chi_{\rho_{V \otimes W}} = \chi_{\rho_V}\cdot \chi_{\rho_W}$, where \(\rho_{V\otimes W}\) denotes the tensor representation of \(\rho_V\) and \(\rho_W\). For \(x = \mathrm{diag}(x_1,\dots,x_n)\),

    \begin{align*}
    \chi_{\mathrm{R}^\lambda}(x)
    \eqalignref{eq:turquoise}
    \chi_{\mathbb{S}^{\widehat\lambda}(\C^n)}(x)\chi_{\det^k}(x)\\
    \eqalignref{eq:char-sc}
    s_{\widehat{\lambda}}(x_1,\dots,x_n)(x_1\cdots x_n)^k\\
    &=
    e_n^k s_{\widehat{\lambda}}(x_1,\dots,x_n)\\
    \eqalignref{eq:schuren}
    s_{\widehat{\lambda}+(k^n)}(x_1,\dots,x_n)\\
    &=
    s_\lambda(x_1,\dots,x_n).
    \end{align*}
\end{proof}

\begin{defin}\label{def:dual-rep}
   Let $V^\vee=\Hom(V,\C)$. For a representation $(\rho,V)$ of a group $G$, the \emph{dual representation}
is the pair $(\rho^\vee,V^\vee)$, where
\begin{align}
    \bigl(\rho^\vee(g)f\bigr)(v)
=
f\bigl(\rho(g^{-1})v\bigr) \label{eq:dual-rep}
\end{align}
for all $g\in G$, $f\in V^\vee$, and $v\in V$.
\end{defin}

If $V$ is finite dimensional, then \cref{eq:dual-rep} implies
\begin{align}
        \rho^\vee(g)=\rho^t(g^{-1}),
\end{align}
where $\rho^t(g^{-1})$ denotes the transpose of $\rho(g^{-1})$.

\begin{prop}[{\cite[Th.~II.4.1]{ABW82}}]
Let $V_n$ be the standard representation. For a partition $\lambda$ with $\ell(\lambda) \le n$, we have
\begin{align}
    \Sc^\lambda (V_n^\vee) \cong (\Sc^\lambda V_n)^\vee.
\end{align}
\end{prop}
\begin{rem}
In \cite{ABW82}, the authors distinguish between Schur and Weyl functors. However, since we are working over a field of characteristic zero, these two functors are isomorphic.
\end{rem}
 
Recall from \cref{prop:char-sc} that $\chi_{\mathbb{S}^\lambda V_n}(x_1, \dots, x_n) = s_\lambda(x_1, \dots, x_n)$.
\begin{prop}\label{prop:char-dual}
    Let $\lambda \in \Par_n$, and $x$ be a shorthand for the variables $x_1, \dots, x_n$. Then, we have 
    \begin{align}
        \chi_{(\mathbb{S}^\lambda V_n)^\vee}(x) =s_\lambda(x_1^{-1}, \dots, x_n^{-1})=  \Gamma(\chi_{\mathbb{S}^\lambda V_n}(x)), \label{eq:cd2}
    \end{align}
    where $\Gamma$ is the map defined in \cref{prop:gamma}.
\end{prop}
\begin{proof}
Let $W=\Sc^\lambda V_n$ and let
$\operatorname{diag}(x) = \operatorname{diag}(x_1,\dots,x_n)$. By \cref{def:dual-rep},
\[
\rho_{W^\vee}(\operatorname{diag}(x))=\rho_W(\operatorname{diag}(x)^{-1})^t.
\]
Therefore,
\[
\begin{aligned}
\chi_{W^\vee}(\operatorname{diag}(x))
&=\mathrm{Tr}\bigl(\rho_W(\operatorname{diag}(x)^{-1})^t\bigr)\\
&=\mathrm{Tr}\bigl(\rho_W(\operatorname{diag}(x)^{-1})\bigr)\\
&=\chi_W(\operatorname{diag}(x)^{-1})\\
&=s_\lambda(x_1^{-1},\dots,x_n^{-1}).
\end{aligned}
\]
The last equality follows from \cref{prop:char-sc}, and the equality with
$\Gamma(\chi_W(x))$ follows from \cref{prop:gamma}.
\end{proof}

Recall from \cref{prop:relsch} that for $\lambda \in \SPar_n$, we have $s_\lambda(x_1^{-1}, \dots, x_n^{-1}) = s_{\lambda^{\vee}}(x_1,\dots,x_n)$, where $\lambda^\vee= ( - \lambda_n, \dots, -\lambda_1) \in \SPar_n$. For this reason, $\lambda^{\vee}$ is called the \emph{dual partition of $\lambda$}. Note that this definition is independent of \(n\) since if we write \(\lambda = (\mu,\nu)\) it is easily checked that \(\lambda^\vee = (\nu,\mu)\).

\section{The infinite dimensional general linear group} \label{sec:repGLinf}
In this section, we restrict our attention to algebraic representations of $\mathrm{GL}_\infty$. Note that the standard (natural) representation $V$ and its restricted dual $V_\vee$ are infinite-dimensional. While the trivial representation (and its finite direct sums) are finite-dimensional, every non-trivial simple algebraic representation of $\mathrm{GL}_\infty$ is infinite-dimensional.

\begin{defin}
    We define the \emph{infinite dimensional general linear group} as
    \begin{align}
        \GL_\infty=\varinjlim_{n\geq1}\GL_n,
    \end{align}
    where the maps for the direct limit are given by
    \begin{align}
    \GL_n \hookrightarrow\GL_{n+1}, \quad M \longmapsto \begin{pmatrix}
            M & 0 \\
            0 & 1
        \end{pmatrix}.
    \end{align}
\end{defin}
Since the natural maps \(\GL_n \hookrightarrow \GL (\C^\infty)\) commute with the arrows in the direct limit, we get an injection \(\GL_\infty \hookrightarrow \GL(\C^\infty)\), and we can identify $\GL_\infty$ with a subgroup of $\GL(\C^\infty)$, which allows us to work with elements of $\GL_\infty$ as infinite matrices. Thus, \(\GL_\infty\) consists of infinite invertible matrices that differ from the identity matrix in only finitely many entries. 

\begin{defin}
    In the following we let $V = \C^\infty$ be the \emph{natural representation} (i.e. the $\GL_\infty$ representation with $\rho$ the injection into \(\GL(\C^\infty)\) previously mentioned). The \emph{restricted dual} of $V$ is
    \begin{align}
        V_\vee = \varinjlim_{n \geq 1} (\C^n)^\vee
    \end{align}
    where the maps for the direct limit are given by
    \begin{align}
        (\C^n)^\vee \hookrightarrow (\C^{n+1})^\vee, \quad
        e_i^\vee \longmapsto e_i^\vee
    \end{align}
    where $e_i^\vee$, $1 \leq i \leq n$ denotes the dual basis to the standard basis of $\C^n$. We may identify each \(e_n^\vee \in (\C^n)^\vee\) with its image in \(V_\vee\) (this makes sense since the maps \((\C^n)^\vee \rightarrow V_\vee \) are injective), then the \(e_n^\vee\), \(n \in \N_{\geq 1}\) form a basis of \(V_\vee\).
\end{defin}

     Note also that each \((\C^n)^\vee\) comes with a representation of \(\GL_n\) (the dual of the standard one) and the maps \((\C^n)^\vee \hookrightarrow (\C^{n + 1})^\vee\) commute with the action of \(\GL_n\) (where \((\C^{n + 1})^\vee\) is given by the restriction representation from \(\GL_{n + 1}\) to \(\GL_n\)). This induces a representation of \(\GL_\infty\) on \(V_\vee\).

\begin{defin}
    For $n,m \in \N$, let $T_{n,m} := V^{\otimes n}\otimes (V_\vee)^{\otimes m}$. An \emph{algebraic representation} of $\GL_\infty$ is a representation that is a subquotient (i.e. a quotient of a subrepresentation) of finite direct sums of $T_{n,m}$.
\end{defin}

Note that algebraic representations of $\GL_\infty$ are the infinite dimensional analogue of rational representations of $\GL_n$. Henceforth, all representations of $\GL_\infty$ we consider will be algebraic unless otherwise indicated.

\begin{defin}
    Let $(V,\rho)$ be a representation of $\GL_\infty$. A \emph{weight vector $v \in  V$} is a nonzero vector, such that there exists a weight $\alpha = (\alpha_1, \alpha_2, \dots )$ with $\alpha_i \in \Z$  and an $r \in \N$ such that $\alpha_i = 0 $ for all $i > r$, and, for every $s\geq r$ and every $g = \operatorname{diag}(g_1, \dots, g_s, 1, 1 , \dots)\in\GL_\infty$, 
    \begin{align*}
        \rho(g) v = g_1 ^{\alpha_1} \cdots g_r^{\alpha_r} v=g_1^{\alpha_1}\cdots g_s^{\alpha_s} v.
    \end{align*}
\end{defin}

\begin{prop}[{\cite[\S 3.1.7]{SS15}}]
    An algebraic representation of $\GL_\infty$ decomposes into weight spaces.
\end{prop}

\begin{prop}[{\cite[Prop. 3.1.4]{SS15}}]
    For every signed partition \((\lambda,\mu)\) there is an irreducible representation \(V_{\lambda,\mu}\) of \(\GL_\infty\). Furthermore, these form an irredundant complete set of irreducible representations.
\end{prop}

For the explicit definition of \(V_{\lambda,\mu}\), see \cite[\S 3.1.2]{SS15}. In particular, we have \(V_{(1),\varnothing} = V\) and \(V_{\varnothing,(1)} = V_\vee\).

We now extend character theory to the representations of \(\GL_\infty\). Unfortunately, there is no obvious way to generalize \cref{def:char} since non-trivial simple algebraic representations are infinite-dimensional. Likewise, we cannot generalize \cref{eq:char} since the weight spaces can be infinite dimensional as well. Thus, we turn to the formalism of the \textit{Grothendieck ring}.

\begin{defin}
    The \emph{representation ring  of the $\GL_\infty$ group}, (or the \emph{Grothendieck ring}), denoted \(R(\GL_\infty)\), is the abelian group generated by isomorphism classes of algebraic $\GL_\infty$ representations $[W]$, modulo the relations
    \begin{align}
        [W] - [W'] - [W''] = 0
    \end{align}
    for all exact sequences \(0 \to W' \to W \to W'' \to 0\) of \(\GL_\infty\) representations. Multiplication in the Grothendieck ring is given by the bilinear extension of 
\begin{align}
    [W] \cdot [U] = [W \otimes U].
\end{align}
\end{defin}

\begin{eg}
    The relation 
    \begin{align*}
        [W] + [U] = [W \oplus U]
    \end{align*}
    is given by the exact sequence \(0 \to W \to W \oplus U  \to U \to 0\).
\end{eg}

\begin{defin}
    Let $W$ be a representation. A \emph{composition series} is a series of subrepresentations
    \[
0 = W_0 \subsetneq W_1 \subsetneq \cdots \subsetneq W_n = W
\]
such that the quotients $W_{i+1}/W_i$, called \emph{composition factors}, are irreducible. 
\end{defin}

\begin{prop}~
    \begin{enumerate}
        \item  Any algebraic representation of \(\GL_\infty\) has a composition series. \label{prop:terminates}
        \item Any two composition series have the same length and, up to reordering, isomorphic composition factors.  \label{prop:jordan-holder}
    \end{enumerate}
\end{prop}

\begin{proof}
    \cref{prop:terminates} is proved in {\cite[Prop. 3.1.5]{SS15}}. \cref{prop:jordan-holder} is a consequence of the Jordan-H\"older theorem for finite length modules. 
\end{proof} 

\begin{prop}\label{prop:comps}
    Let $W$ be a representation and $0 = W_0 \subsetneq W_1 \subsetneq \cdots \subsetneq W_n = W$ be a composition series. Then, 
    \begin{align}
        [W] = [W_n/W_{n - 1}] + \cdots + [W_1/W_0].
    \end{align}
Thus, the simple representations span $R(\GL_\infty)$ as an abelian group.
\end{prop}

\begin{proof}
    Using induction on the length of the composition series, and noting that we have the exact sequence \(0 \to W_{n - 1} \to W_n \to W_n/W_{n - 1} \to 0\),
    we get
    \[[W_n] = [W_n/W_{n -1}] + [W_{n - 1}]  =   [W_n/W_{n -1}] + [W_{n - 1}/W_{n - 2}] + \cdots + [W_1/W_0],\]
   where the second equality is given by using the induction hypothesis on \(W_{n -1}\), which has length \(n - 1\).
\end{proof}

\cref{prop:comps} can be strengthened to show that these simple representations are also linearly independent.

\begin{prop}[{\cite[C-4.93]{Rot17}}]
    The isomorphism classes of the simple representations form a basis of \(R(\GL_\infty)\).
\end{prop}

    Note that it may happen for two non-isomorphic \(\GL_\infty\) representations \(W,U\) that \([W] = [U]\) in \(R(\GL_\infty)\) since non-semisimple representations of $\GL_\infty$ exist (see \cite[1.2.2]{SS15}). Indeed, take any non-semisimple representation \(W\) and let \(W_1,\dots,W_n\) be its composition factors, then
\[
[W] = [W_1] + \cdots + [W_n] = [W_1 \oplus \cdots \oplus W_n].
\]
But clearly \(W\) is not isomorphic to \(W_1 \oplus \cdots \oplus W_n\) since that would make \(W\) semisimple.

\begin{theo}[{\cite[\S 3.2]{SS19}, \cite[§2]{Koi89}}]
There is a ring isomorphism
\begin{align*}
    \mathrm{ch}\colon R(\GL_\infty) \longrightarrow \LSym
\end{align*}
such that \(\mathrm{ch}(V_{\lambda,\mu}) = s_{\lambda, \mu}\). 
\end{theo}

\begin{defin} Given a representation \(V\) of \(\GL_\infty\), we define its \emph{character} \(\chi_V \in \LSym\) as
\begin{align}
    \chi_V = \mathrm{ch}([V]).
\end{align}
\end{defin}

The characters allow us to obtain information about the group representations from the structure of \(\LSym\). 

 In general, \(s_{\lambda,\mu}\neq s_{\lambda,\varnothing}s_{\varnothing,\mu}\), which we illustrate in \cref{ex:schur} below. 
\begin{eg}\label{ex:schur}
   Let \(\lambda = \mu = (1)\). We have $ s_{(1),\varnothing}=s_1(X)$ and $s_{\varnothing,(1)}=s_1(Y)$.
   Therefore, the product $s_{(1),\varnothing}s_{\varnothing,(1)}$ is given by $s_1(X)s_1(Y)$. 
   Applying \cref{def:lsf} with $\lambda=\mu=(1)$, we get
    \[
    s_{(1),(1)}=s_1(X)s_1(Y)-1. 
    \]
    Hence, we have
    \begin{equation}\label{eq:s11}
        s_{(1),\varnothing}s_{\varnothing,(1)}
        =
        s_{(1),(1)}+1,
    \end{equation}
    so in particular, $s_{(1),(1)}
        \neq
        s_{(1),\varnothing}s_{\varnothing,(1)}$.
    \end{eg}
    
    Equation \cref{eq:s11} implies that the composition factors of \(V \otimes V_\vee = V_{(1),\varnothing} \otimes V_{\varnothing,(1)}\) are \(V_{(1),(1)}\) and \(V_{\varnothing,\varnothing} = \C\) (the trivial representation). However, \(V \otimes V_\vee\) is not semisimple \cite[\S 3.1.1]{SS15}, thus, it cannot be isomorphic to the semisimple representation \(V_{(1),(1)} \oplus \C\).

   In the analogous finite-dimensional case $n>1$, applying $\rho_n$ to
\cref{eq:s11} gives
\[
\chi_{V_n\otimes V_n^\vee}
=
s_{(1),(1)}^{(n)}+1.
\]
Under the natural $\GL_n$-equivariant isomorphism
\[
V_n\otimes V_n^\vee \longrightarrow \operatorname{End}(V_n),
\qquad
v\otimes f \longmapsto \bigl(w\mapsto f(w)v\bigr),
\]
the action on $\operatorname{End}(V_n)$ is conjugation. Since trace is
invariant under conjugation, the decomposition
\[
\operatorname{End}(V_n)
=
\mathfrak{sl}_n\oplus \C\operatorname{Id}_{V_n}
\]
is a decomposition into $\GL_n$-subrepresentations. Hence
\[
V_n\otimes V_n^\vee
\cong
\mathfrak{sl}_n\oplus\C.
\]
Thus, unlike $V\otimes V_\vee$, the finite-dimensional tensor product
splits as a direct sum.

\section*{Acknowledgments}
This research was partially supported by the Natural Sciences and Engineering Research Council of Canada (NSERC), funding reference number RGPIN-2023-03842. The authors extend their gratitude to the Fields Institute for Research in Mathematical Sciences for hosting this project as part of the 2026 Fields Undergraduate Summer Research Program and for its generous support. The authors are especially grateful to the project supervisors, Alistair Savage and Yaolong Shen, for their guidance throughout this project.

The authors used OpenAI’s ChatGPT-5.6 Sol to assist with proofreading the manuscript and to identify possible gaps and suggest corrections. All such suggestions were independently checked and revised by the authors, who assume full responsibility for the correctness and content of the paper.

\bibliographystyle{alphaurl}
\bibliography{LaurentSym}
\end{document}